\documentclass[11pt,letter]{amsart}

\usepackage[utf8]{inputenx}

\usepackage{appendix}
\usepackage{lmodern}
\usepackage{textcomp}
\usepackage[english]{babel}
\usepackage[T1]{fontenc}
\usepackage{bm}
\usepackage{amssymb,amsthm,amsmath}
\usepackage{indentfirst}
\usepackage{color}
\usepackage{graphicx}
\usepackage{enumitem}
\usepackage{caption}
\usepackage{empheq}
\usepackage{xcolor}
\usepackage{color}
\usepackage{amssymb}
\usepackage{booktabs}
\usepackage{graphicx}
\usepackage[font=small,labelfont=bf,labelsep=period,tableposition=top]{caption}
\usepackage{hyperref}

\newtheorem{theorem}{Theorem}[section]
\newtheorem{lemma}[theorem]{Lemma}
\newtheorem{proposition}[theorem]{Proposition}
\newtheorem{remark}[theorem]{Remark}
\newtheorem{definition}{Definition}[section]
\newtheorem{corollary}[theorem]{Corollary}

\usepackage{ifthen}
\provideboolean{shownotes}\setboolean{shownotes}{true}
\newcommand{\margnote}[1]{\ifthenelse{\boolean{shownotes}}
	{\marginpar{\raggedright\tiny\texttt{#1}}}{}}

\hypersetup
{   bookmarks=false,
	colorlinks=true,
	urlcolor=blue,
	citecolor=blue,
	pdfauthor = {Andrea Aspri},
	pdftitle = {dislocation},
	pdfcreator = {LaTeX, TeXmaker, pdfLaTeX, Hyperref}
}
	
\usepackage{verbatim}

\newcommand{\CC}{\mathbb{C}}
\newcommand{\RR}{\mathbb{R}}

\newcommand{\bu}{\bm{u}}
\newcommand{\bn}{\bm{n}}
\newcommand{\bg}{\bm{g}}
\newcommand{\bx}{\bm{x}}
\newcommand{\cD}{\mathcal{D}}
\newcommand{\by}{\bm{y}}

\definecolor{myorange}{rgb}{0.80, 0.60, 0.25}

\begin{document}
	
\title[Elastic dislocations]{Analysis of a model of elastic dislocations in geophysics}

\author[A. Aspri {\em et al.}]{Andrea Aspri}
\address{ Johann Radon Institute for Computational and Applied Mathematics (RICAM)}
\email{andrea.aspri@ricam.oeaw.ac.at}

\author[]{Elena Beretta}
\address{Dipartimento di Matematica, Politecnico di Milano and Department of Mathematics, NYU-Abu Dhabi}
\email{elena.beretta@polimi.it}

\author[]{Anna L. Mazzucato}
\address{Department of Mathematics, Penn State University}
\email{alm24@psu.edu}

\author[]{Maarten V. De Hoop}
\address{Department of Computational and Applied Mathematics and Department of Earth, Environmental, and Planetary Sciences, Rice University}
\email{mdehoop@rice.edu}

\date{\today}



\begin{abstract}
We analyze a mathematical model of {\it elastic dislocations} with applications to geophysics, where by an elastic dislocation we mean an open, oriented Lipschitz surface in the interior of an elastic solid, across which there is a 
discontinuity  of the displacement. We model the Earth as an infinite, isotropic, inhomogeneous, elastic medium occupying a half space, and assume only Lipschitz continuity of the Lam\'e parameters. We study the well posedness of very weak solutions to the forward problem of determining the displacement by imposing traction-free boundary conditions at the surface, continuity of the traction and a given jump on the displacement across the fault. We employ suitable weighted Sobolev spaces for the analysis. We utilize the well posedness of the forward problem and unique-continuation arguments to establish uniqueness in the inverse problem of determining the dislocation surface and the displacement jump from measuring the displacement at the surface of the Earth. Uniqueness holds for tangential or normal jumps and under some geometric conditions on the surface.
\end{abstract}

\maketitle

\section{Introduction}

In this paper,  we analyze a mathematical model of {\it elastic dislocations} with applications to geophysics,  see for 
example \cite{Bonafede-Rivalta, Eshelby,Rivalta-Mangiavillano-Bonafede,Segall10,Zwieten-Hanssen-Gutierrez}. An 
{\it elastic dislocation} is an open, oriented surface in the interior of an elastic solid, across which there is a
discontinuity of the displacement. It describes a fault plane undergoing slip over a limited area, a thin intrusion such as a dyke, or a crack the  faces  of which slide over one another or separate by the action of an applied stress. 
An elastic dislocation for which  the displacement discontinuity varies from point to point of the internal surface is 
called  a {\it Somigliana dislocation}, while in the particular case of a constant displacement discontinuity it is known 
as a  {\it Volterra dislocation} (see \cite{Eshelby,Volterra}).

We model the dislocation by an open, oriented Lipschitz surface $S$ with Lipschitz boundary $\partial S$ such that 
$\overline{S}\subset \mathbb{R}^3_-$. In particular, we assume that the dislocation is at positive distance from the 
surface of the Earth, identified with the plane $\{x_3=0\}$. We orient $S$ by choosing a unit normal vector $\bn$. 
In geophysical applications, one can assume that the closure $\overline S$ is compact. We assume the Earth's interior 
to be an  isotropic and inhomogeneous  infinite elastic medium. 
In the regime of  small-amplitude deformations, we are led to study  a boundary-value/transmission problem in a 
half space for the system of linearized elasticity:
\begin{equation}\label{eq: transm_problem_intro}
			\begin{cases}
				\textrm{\textup{div}}\, (\mathbb{C}\widehat{\nabla}\bm{u})=\bm{0}, & \textrm{in}\,\, 
				\mathbb{R}^3_- \setminus \overline{S},\\
				(\mathbb{C}\widehat{\nabla}\bm{u})\bm{e}_3=\bm{0}, & \textrm{on}\,\, \{x_3=0\},\\
				[\bm{u}]_{S}=\bm{g}, \\
				[(\mathbb{C}\widehat{\nabla}\bm{u})\bm{n}]_{S}=\bm{0}, \\
			\end{cases}
\end{equation}
where $\mathbb{C}$ is the Lam\'e tensor with non-constant coefficients of Lipschitz class, satisfying the usual 
strong convexity assumption, $\bm{u}$ is the displacement field, $\bm{e}_3=(0,0,1)$ is the unit normal vector on 
$\{x_3=0\}$,  and  $\bm{g}$ the displacement jump across  the dislocation  $S$. As customary, we denote the jump 
of a function or tensor field 
$\bm{f}$ across $S$ by $[\bm{f}]_s:=\bm{f}_+-\bm{f}_-$, where $\pm$ denotes a non-tangential limit to each side of the oriented surface 
$S$, $S_+$ and $S_-$, where $S_+$ is by convention the side determined by $\bn$. 
The {\em direct} or {\em forward problem} consists, knowing $\CC$, $S$, and $\bg$, in finding $\bu$ solution of 
\eqref{eq: transm_problem_intro}. The {\em inverse problem} consists in determining $S$ and $\bg$ from 
measurements made on $\bu$. In seismology and geophysics, the data is typically in the form of measurements taken at the 
surface of the Earth. For the dislocation problem, since the solution is traction-free at the boundary, these data 
consist in measurements of the displacement at the surface induced by the jump $\bg$ at the dislocation.
To be more specific, we investigate the inverse problem of determining  dislocations $S$ caused by a tangential slip 
along the dislocation surface (the case of a purely normal jump across the surface can also be included in our analysis) from surface 
measurements of  the displacement $\bm{u}$ on some bounded open portion $\Sigma$ of $\{x_3=0\}$. 
One of the main results of this work is the unique determination of the dislocation surface $S$ and the slip strength $\bg$ 
from knowledge of $\bu$ on $\Sigma$, under some geometric conditions on $S$. These conditions are satisfied, for 
example, by polyhedral surfaces.
\begin{figure}[h]
	\centering
	\includegraphics[scale=0.5]{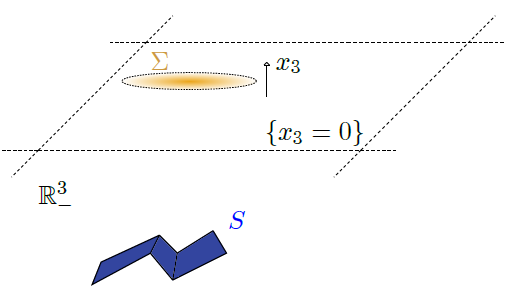}
	\caption{Geometrical setting.}
\end{figure}
{The inverse problem investigated in this paper is of particular interest in and motivated by applications to geophysics. In
fact, the analysis of coseismic deformation through the inverse slip and dislocation problem might enhance the 
understanding of failure at faults and microseismicity, see for example \cite{Evans-Meade} and references therein. 
Understanding the earthquake source process and how earthquake sequences evolve requires accurate estimations of how stress changes due to an earthquake. Models of stress change induced by earthquakes have been performed using solutions that assume homogeneous slip \cite{Okada}. Cohen \cite{Cohen} introduced simplified formulas in the case of a
long dip-slip fault that ruptures the surface. To mitigate the error from the assumption of homogeneous slip, faults have commonly been divided into many smaller patches with different slip \cite{Evans-Meade,Jiang-Wang-Wu-Che-Shen}. A recent triangular fault patch
slip solution \cite{Nikkhoo-Walter} has triggered studies with more general fault morphologies.
The basic model -- assuming a patched fault but a generally
heterogeneous surrounding medium -- analyzed in this paper has been
widely used in the analysis of different earthquakes over the past two
decades \cite{Cambiotti-Zhou-Sparacino-Sabadini-Sun,Deloius-Nocquet-Vallee,Evans-Meade,Jiang-Wang-Wu-Che-Shen,Johnson,Serpelloni-Anderlini-Belardinelli,Simons-Fialko-Rivera,Trasatti-Hyriakopoulos-Chini,Walker-Bergman-Szeliga-Fielding}. Here, we note that
the inverse coseismic slip problem is often combined with the analysis
of seismic data to colocate ruptures and faults. Furthermore, we note that Simons et al. \cite{Simons-Fialko-Rivera} consider also a layered half space containing discontinuities while we assume smoothly varying Lam\'{e} parameters.
}

In order to investigate the inverse problem, one needs to have a good well-posedness theory for the forward 
problem.  This problem is less studied than more classical transmission problems in bounded domains, and there are 
 some additional technical difficulties that need to be overcome. For starters, 
the solution is discontinuous across the interface and the quadratic variational form needs to be properly 
augmented to obtain coercivity for weak solutions. In addition, the problem is naturally posed in an unbounded 
domain, a half space, which leads us to employ suitable weighted Sobolev spaces. 
Here, and throughout the paper,  we denote standard $L^2$-based Sobolev spaces with $H^s$, $s\in \RR$. The notation for the weighted spaces is discussed in Section \ref{sec: notation and functional setting}.

If $ \bm{g}$ belongs to the space $H^{1/2}(S)$ and has compact support in $S$, then the transmission 
problem \eqref{eq: transm_problem_intro} admits  a unique variational solution $\bu$ on 
$\mathbb{R}^3_- \setminus \overline{S}$ in a suitable weighted Sobolev space that takes into account the conditions at infinity 
(see \cite{Volkov-Voisin-Ionescu} for the case of constant coefficients). 
In fact, $\bu$ can be expressed as a double layer potential on the dislocation $S$, i.e.,
\begin{equation*}
			\bm{u}(\bm{y})=-\int\limits_{S}\left[(\mathbb{C}(\bm{x})\widehat{\nabla}_{\bm{x}}\bf N (\bm{x},
			\bm{y})) \bm{n}(\bm{x})\right]^T \bm{g}(\bm{x})\, d\sigma(\bm{x}),
		\end{equation*}
where $\mathbf{N}$ is the matrix-valued Neumann function 
\begin{equation*}
	\begin{cases}
		\textrm{\textup{div}}\,(\mathbb{C}(\bm{x})\widehat{\nabla}\bf N(\bm{x},\bm{y}))=\delta_{\bm{y}}
		(\bm{x})\bf I ,& \textrm{in}\, \, \mathbb{R}^3_-,\\
		(\mathbb{C}(\bm{x})\widehat{\nabla}\bf N(\bm{x},\bm{y}))\bm{e}_3=\bm{0}, & \textrm{on}\, \, \{x_3=0\},
		\\
	\end{cases}
	\end{equation*}
satisfying certain decay conditions at infinity.
The solution has then locally $H^1$-regularity from standard regularity results for potential theory on Lipschitz 
surfaces (see  e.g. \cite{Mitrea-Taylor} and references therein).  
For the application to the inverse problem, we cannot restrict the support of $\bg$, which may coincide with 
$\overline S$. In this case, even when $\bm{g}\in H^{1/2}(S)$, the existence of a variational solution with $H^1$ 
regularity locally in the complement of the dislocation is not guaranteed anymore. For example, if $S$ is a rectangular Volterra dislocation, the double layer potential  blows up logarithmically at the vertices of the rectangle, as already observed  in \cite{Okada}. 

Hence, establishing the well-posedness of the problem \eqref{eq: transm_problem_intro} in  full generality is rather 
delicate and not covered by results in the literature. We devote the first half of the paper to investigating this 
problem. We find it convenient to reformulate the transmission problem as an equivalent source problem in the 
whole half space  $\RR^3_-$:
\begin{equation}\label{pb: problem_intro}
		\begin{cases}
				\textrm{\textup{div}}\, (\mathbb{C}\widehat{\nabla}\bm{u})=\bm{f}_S ,& \textrm{in}\,\, 
				\mathbb{R}^3_-, \\
				(\mathbb{C}\widehat{\nabla}\bm{u})\bm{e}_3=\bm{0}, & \textrm{on}\,\, \{x_3=0\},
		\end{cases}
	\end{equation} 	
with the source term given by
	\begin{equation}\label{eq: source_term_intro}
			\bm{f}_{S}=\textrm{div}(\mathbb{C}(\bm{g}\otimes \bm{n})\delta_{S}),
	\end{equation}		
where $\delta_{S}$ represents the Dirac measure concentrated on $S$.
In the first part of the paper we show well-posedness of this problem in the weighted Sobolev space 
$H^{{1}/{2}-\varepsilon}_{-{1}/{2}-\varepsilon}(\mathbb{R}^3_-)$ with $\varepsilon>0$, that is in the context of 
very weak solutions. For a definition of this weighted space, see Section \ref{sec: notation and functional setting}.  
We stress that  this elliptic regularity result is optimal as the source \eqref{eq: source_term_intro} is a distribution with compact 
support belonging to  $H^{-{3}/{2}-\varepsilon}(\mathbb{R}^3_-)$.
The above result leads to  a mathematically rigorous analysis of the elastic dislocation 
model in the geophysical framework with minimal assumptions on the elastic coefficients. We follow the  approach 
of \cite{Lions-Magenes} for the case of regular coefficients on bounded domains, which is based on duality 
arguments and interpolation, adapted to the framework of weighted spaces in a half space in \cite{Hanouzet} and 
\cite{Amrouche-Dambrine-Raudin} when the  source terms are integrable on $\RR^3_-$. We extend it to the case of 
more singular source terms for the system of  linearized elasticity with Lipschitz coefficients. 
We also extend the representation formula of the solution as a double layer potential. From regularity results  for potential theory on Lipschitz surfaces (see again  \cite{Mitrea-Taylor}), it follows that the displacement lies in $H^s(\mathbb{R}^3_-\setminus S)$ locally for $s<1$, although it fails to belong to $H^1$ even locally near the dislocation surface. 
The analysis of the forward problem and, in particular, the hypothesis that 
$\textrm{supp}(\bm{g})=\overline{S}$ are essential in the second part of the paper to investigate the inverse problem of 
determining  dislocations $S$ caused by a tangential or normal slip along the dislocation surface from surface measurements of the 
displacement $\bm{u}$ on some bounded open portion $\Sigma$ of $\{x_3=0\}$. 
We prove, by using unique-continuation properties of solutions to the Lam\'e system with Lipschitz 
parameters (see \cite{Lin-Nakamura-Wang}), that  one surface measurement of the displacement field is sufficient 
to recover uniquely both the dislocation surface and the slip, assuming some geometric conditions on $S$ --- for instance $S$ can be piece-wise linear {and a graph with respect to a fixed, but arbitrary, coordinate frame} --- and assuming the slip field is  either purely tangential (corresponding to the case of a fault, the two sides of which  slide one  over the other) or directed 
in the normal direction (corresponding to a crack, the two sides of which separate from one another).  
It is important to note that in three space dimensions, the unique continuation property may not hold  for a second-order elliptic operator, if its coefficients are only in H\"older classes $C^{0,\alpha}$, $\alpha<1$, and not Lipschitz continuous (see \cite{Miller,Plis}). In the framework of the Lam\'e operator with Lipschitz continuous coefficients very recent results on unique continuation are in \cite{Koch-Lin-Wang,Lin-Nakamura-Wang}. Therefore, the regularity assumption on the Lam\'e parameters that we impose appears optimal in this context.

In this paper,  we do not tackle the problem of determining quantitative stability estimates for the inverse problem, which we 
leave for future work. For non-quantitative stability estimates in the setting of constant Lam\'e coefficients, see \cite{Triki-Volkov}. However, we expect to be able to prove Lipschitz stability estimates for piece-wise linear 
dislocations in terms of the boundary data, thanks to the presence of singularities at the corners of the dislocation and generalizing quantitative unique continuation estimates obtained in \cite{Morassi-Rosset}.  
In fact, in \cite{Beretta-Francini-Vessella} the authors are able to prove a Lipschitz stability estimate for linear cracks 
in a two dimensional elastic homogeneous medium, taking advantage of the presence of singularities at the endpoint 
of the crack. We refer to \cite{Beretta-Francini-Kim-Lee} for a reconstruction algorithm in this setting.
Concerning the reconstruction of dislocations and tangential slips in three space dimensions, several algorithms have 
been proposed, mostly related to the case of rectangular dislocations and homogeneous Lam\'e parameters. 
In the mathematics literature, we refer to \cite{Volkov-Voisin-Ionescu}, where the authors implement an iterative 
method for detecting the plane containing a fault and the tangential slip, supposed to be unidirectional, in the 
presence of only a finite number of surface measurements,  via a constrained minimization of a suitable misfit 
functional.  In the geophysics literature, we mention \cite{Arnodottir-Segall}, where a two-step algorithm is discussed,  first assuming a uniform slip and using a nonlinear, quasi-Newton method to locate the dislocation surface, and then recovering a non-uniform slip along the surface. We mention also the works   \cite{Fukahata-Wright,Zhou-Cambiotti-Sun-Sabadini},  \cite{Evans-Meade} (and references therein), where a Bayesian framework and a sparsity-promoting, 
state-vector regularization method are  employed, respectively.

The paper is organized as follows; in Section \ref{sec: notation and functional setting}  we introduce needed 
notation, and  recall the definition of certain weighted Sobolev spaces and their relevant properties useful for our 
analysis.  Section \ref{sec: the direct problem} is devoted to the study of the well posedness of problem \eqref{pb: 
problem_intro} and to an analysis of the regularity of the solution in a neighborhood of the dislocation surface. Finally, in Section 
\ref{sec: inverse problem} we prove a uniqueness result for the inverse problem of determining a piece-wise linear 
dislocation surface and its slip. We present an explicit calculation of the singularities in the displacement for the Volterra dislocation  in the Appendix.

\section*{Acknowledgments}
The authors thank C. Amrouche, S. Salsa, and M. Taylor for useful discussion and for suggesting relevant literature. {They also thank the anonimous referees for their careful reading of this work and useful suggestions}. A. Aspri and A. Mazzucato  thank the 
Departments of Mathematics at NYU-Abu Dhabi and at Politecnico of Milan for their hospitality. A. Mazzucato was partially 
supported by the US National Science Foundation Grant DMS-1615457. M. V. de Hoop gratefully acknowledges support from the Simons Foundation under the MATH + X program, the National Science Foundation under grant DMS-1815143, and the corporate members of the Geo-Mathematical Group at Rice University.

\section{Notation and Functional Setting}\label{sec: notation and functional setting}
In this section we introduce the  weighted Sobolev spaces used for the analysis of the forward problem. We begin by recalling some standard, but needed, notation.

We denote scalar quantities in italics, e.g. $\lambda, \mu, \nu$,
points and vectors in bold italics, e.g.  $\bm{x}, \bm{y}, \bm{z}$ and $\bm{u}, \bm{v}, \bm{w}$, matrices and 
second-order tensors in boldface, e.g.  $\mathbf{A}, \mathbf{B}, \mathbf{C}$, and fourth-order tensors in 
blackboard face, e.g.  $\mathbb{A}, \mathbb{B}, \mathbb{C}$. 

 We indicate the symmetric part of a second-order tensor $\mathbf{A}$ by
	\begin{equation*}
		\widehat{\mathbf{A}}=\tfrac{1}{2}\left(\mathbf{A}+\mathbf{A}^T\right),
	\end{equation*}
where $\mathbf{A}^T$ is its transpose.
We use the standard notation for inner product between two vectors $\bm{u}$ and $\bm{v}$,
$\bm{u}\cdot \bm{v}=\sum_{i} u_{i} v_{i}$, and between second-order tensors $\mathbf{A}:
\mathbf{B}=\sum_{i,j}a_{ij} b_{ij}$. 
The tensor product of two vectors $\bm{u}$ and $\bm{v}$ is denoted by $[\bm{u}\otimes \bm{v}]_{ij}=u_i v_j$. 
With $|\mathbf{A}|$ we denote the norm induced by the inner product between second-order tensors, that is,
	\begin{equation*}
		|\mathbf{A}|=\sqrt{\mathbf{A}:\mathbf{A}}.
	\end{equation*}
The vectors $\bm{e}_1, \bm{e}_2$, and $\bm{e}_3$ represent the standard orthonormal basis of $\mathbb{R}^3$, 
and we denote the lower half space by: 
	\begin{equation*}
		\{\bm{x}=(x_1,x_2,x_3)\in \mathbb{R}^3\,:\,x_3< 0\}=\mathbb{R}^3_-.
	\end{equation*}
We canonically identify its boundary   $\{(x_1,x_2,x_3)\in \mathbb{R}^3\,:\, x_3=0 \}$ with $\mathbb{R}^2$, and 
denote a point in $\RR^2$ by $\bx'=(x_1,x_2)$.	
The set $B_r(\bm{x})$ represents the ball  of center $\bm{x}$ and radius $r$ and $B^-_r(\bm{x})$, 
$B^+_r(\bm{x})$ the lower and upper half balls, respectively.
With $B'_{r}(\bm{x'})$ we mean the disk of center $\bm{x}'$ and radius $r$, namely 	
	\begin{equation*}
		B'_{r}(\bm{x'})=\{\bm{y}'\in\mathbb{R}^2\, :\, (y_1-x_1)^2+(y_2-x_2)^2<r^2\}.
	\end{equation*}
Finally, we follow the standard multi-index notation for partial derivatives {and throughout the paper, we use the notation $\langle,\rangle_{(X',X)}$ to denote the duality pairing between a {Banach} space $X$ and its dual $X'$.}

\subsubsection*{{\bf Weighted Sobolev Spaces.}}

We next define the weighted Sobolev spaces utilized in our work. We briefly recall their main properties, referring 
the reader for instance to 
\cite{Amrouche-Bonzom,Amrouche-Dambrine-Raudin,Amrouche-Girault-Giroire,Amrouche-Necasova,Amrouche-Necasova-Raudin,Hanouzet} and references 
therein for a more in-depth discussion.

As usual, we begin by defining Sobolev spaces of integer regularity index, and we deal only with $L^2$-based spaces 
to avoid unnecessary technicalities. Throughout, we denote $\mathcal{D}(\Omega)=C^{\infty}_0(\Omega)$ for any 
non-empty set $\Omega$. 
{For our purposes, it will be sufficient to consider the cases $\Omega=\RR^3$, $\Omega=\RR^3_-$, or $\Omega=\RR^3_-\setminus K$, where $K$ is a compact set.}

\begin{definition}[Weighted Sobolev spaces]
Let
	\begin{equation}\label{eq: weight}
		\varrho(\bx)=(1+|\bm{x}|^2)^{1/2}.
	\end{equation}
Let $\Omega$ be a domain in $\RR^3$. For $m\in \mathbb{Z}_+$, $\alpha\in \mathbb{R}$, we define 
	\begin{equation}
	 H^m_{\alpha}(\Omega)=\Bigg\{ f\in \mathcal{D}'(\Omega);\,\, 0\leq |\kappa|\leq m,\,\, \varrho^{\alpha-m+|
	 \kappa|}\,\partial^{\kappa}f \in L^2(\Omega) \Bigg\},
	\end{equation}		
where $\mathcal{D}'(\Omega)$ denotes the dual space of $\mathcal{D}(\Omega)$, that is, the space of distributions on $\Omega$.
\end{definition}
{These function spaces are Hilbert spaces equipped with the inner product
	\begin{equation*}
		(f,g)_{H^m_{\alpha}(\Omega)}=\sum_{|\kappa|=0}^{m}\int_{\Omega}\varrho^{2(\alpha-m+|\kappa|)}\partial^{\kappa}f \partial^{\kappa}g\, d\bm{x}, 
	\end{equation*}
which induces the norm}
	\begin{equation*}
		\|f\|^2_{H^m_{\alpha}({\Omega})}=\sum_{|\kappa|=0}^{m}\|\varrho^{\alpha-m+|\kappa|}\, 
		\partial^{\kappa}f\|^2_{L^2({\Omega})}.
	\end{equation*}
If $\bu$ is a vector function, then we will say $\bu\in H^m_\alpha$ if each component belongs to that space. 
We similarly define weighted spaces in $\RR^2$, replacing $\bx$ with $\bx'$ and  $\varrho$ with $\varrho':=(1+|
\bm{x}'|^2)^{1/2}$. 
With slight abuse of notation, we will refer to $\rho$, $\rho^\alpha$, or simply to $\alpha$, as the ``weight''.   

\begin{remark}
If $\Omega$ is a bounded domain, then the weighted spaces coincide with the usual Sobolev spaces.

We also stress that the spaces $H^m_\alpha$ do not fall into the typical framework of weighted spaces for the analysis on singular spaces, such as conormal Sobolev spaces \cite{Melrose}. In particular, it does not seem possible to reduce to standard Sobolev spaces by conjugation with an appropriate power of the weight.
\end{remark}

{From this point on, we focus on the cases, $\Omega=\mathbb{R}^3$ and $\Omega=\mathbb{R}^3_-$, given that in a neighbourhood of the fault we can utilize the machinery of standard Sobolev spaces by the above remark.} As with the standard Sobolev spaces, one can show that the space $\mathcal{D}({\RR^3})$ is dense in 
$H^m_{\alpha}({\RR^3})$, while $\mathcal{D}(\overline {\RR^3_-})$ is dense in $H^m_{\alpha}(\RR^3_-)$. 
We hence define the space 
\begin{equation}\label{eq: H m,alpha,0}
\mathring{H}^m_{\alpha}(\mathbb{R}^3_-)=\overline{\mathcal{D}(\mathbb{R}^3_-)}^{\|\cdot
\|_{H^m_{\alpha}(\mathbb{R}^3_-)}},
\end{equation}
which is a proper subset of $H^m_{\alpha}(\RR^3_-)$,  while $H^m_\alpha({\RR^3})=\mathring{H}^m_{\alpha}({\RR^3})$.

We define $H^{-m}_{-\alpha}({\RR^3})$ as the dual space to $H^m_\alpha({\RR^3})$.
Similarly, we define  $H^{-m}_{-\alpha}(\mathbb{R}^3_-)$ as the dual space to
$\mathring{H}^m_\alpha(\mathbb{R}^3_-)$.  Both are  spaces  of distributions.

Before introducing fractional spaces and discuss trace results, we recall some basic properties of the spaces 
$H^m_\alpha$. For simplicity, we state most of these properties, such as Poincar\'e's and Korn's inequality  only for the spaces $H^m_\alpha$ that we actually employ in our work, which are $H^1_0(\RR^3_-)$ and $H^2_1(\RR^3_-)$.
The following  weighted Poincar\'e-type inequality 
holds (see for  example \cite{Amrouche-Necasova,Hanouzet}):
\begin{equation}\label{eq: Poincare}
\begin{aligned}
\|f\|_{H^{1}_{0}(\mathbb{R}^3_-)}&\leq C_1 \|\nabla f\|_{L^2(\mathbb{R}^3_-)},\\
\|f\|_{H^{2}_{1}(\mathbb{R}^3_-)}&\leq C_2 \|\varrho\partial^2 f\|_{L^2(\mathbb{R}^3_-)}.\\
\end{aligned}
\end{equation}
Similarly, if a vector field $\bm{u}$ has square-integrable deformation tensor $\widehat{\nabla}\bm{u}$ and 
belongs to the space $H^0_{-1}(\mathbb{R}^3_-)$, then $\bu\in H^1_0(\RR^3_-)$ by a weighted Korn-type inequality 
\cite[Theorem 2.10]{Amrouche-Dambrine-Raudin} (see also \cite{Kondracev-Oleinik}) and the Poincar\'e's inequality above:
\begin{equation}\label{eq: Korn}
	 \|\bu\|_{H^1_0(\RR^3_-)}\leq C\, \|\nabla \bu\|_{L^2(\RR^3_-)} \leq C\, \|\widehat{\nabla}\bm{u}\|_{L^2(\mathbb{R}^3_-)}.
\end{equation}

For any $\gamma\in\mathbb{R}$ such that $(3/2+\alpha-\gamma)\notin \{1,\cdots,m\}$, with
$m\in\mathbb{N}$, the mapping 
	\begin{equation}\label{map: isomorph}
				\bm{u}\in H^m_{\alpha}({\mathbb{R}^3_-})\to \varrho^{\gamma}\bm{u}\in H^m_{\alpha-\gamma}
				({\mathbb{R}^3_-})
	\end{equation}
is an isomorphism; for any $\ell\in \mathbb{N}^3$, the mapping 
	\begin{equation}\label{map: continuity}
                 \bm{u}\in H^m_{\alpha}({\mathbb{R}^3_-})\to \partial^{\ell}\bm{u}\in 
                  H^{m-|\ell|} _{\alpha} ({\mathbb{R}^3_-})
	\end{equation}
is continuous. For more details {and extention of all the previous properties in $\mathbb{R}^2$} see for example \cite{Amrouche-Necasova-Raudin,Hanouzet}.

{We next define fractional spaces on $\RR^3$ for $0<s<2$. We stress that} the same definition applies in $\RR^2$ for $0<s<1$ with the obvious change of
notation (e.g. $\varrho$ is replaced by $\varrho'$). In both cases, we assume $\alpha\geq -1$ and $d/2+\alpha\ne s$, $d=2,3$.   

For $0<s<1$, we define the weighted fractional space $H^s_\alpha(\RR^3)$ as:
\begin{equation} \label{eq:def_fractionalSobolev1}
    H^s_\alpha(\RR^3) = \Bigg\{ f\in \cD'(\RR^3) ;\, \varrho^{\alpha-s}\, f \in L^2(\RR^3), \, |f|_s <\infty\Bigg\},
\end{equation}
where $|\cdot |_s$ denotes a weighted Gagliardo seminorm:
\[
     |f|_s := \iint_{\RR^3\times \RR^3} \frac{|\varrho^\alpha(\bx)\,f(\bx)-\varrho^\alpha(\by)\, f(\by)|^2}{|\bx-
      \by|^{3+2s}}\, d\bx\, d\by.
\]
For $1<s<2$, we define $H^s_\alpha(\RR^3)$ as:
\begin{equation} \label{eq:def_fractionalSobolev2}
    H^s_\alpha(\RR^3) = \Bigg \{ f\in \cD'(\RR^3) ;\,  0\leq |\kappa|\leq [s]-1,\,\, \varrho^{\alpha-s+|\kappa|}\,
    \partial^{\kappa} f \in L^2(\RR^3) , \;  \partial^{[s]} f \in H^{s-[s]}_{\alpha}(\RR^3)\Bigg\},
\end{equation}
with $[s]$ the greatest integer less than or equal to $s$.

\begin{remark}
In general, a logarithm correction to the weight is needed to properly define the spaces unless some conditions on 
the indices $s$ and $\alpha$ for a given dimension $d$  are satisfied (see e.g \cite{Amrouche-Dambrine-Raudin}).  
These conditions are met if $s$ and $\alpha$ are taken as assumed above.
\end{remark}

Fractional spaces on $\RR^3_-$ can be defined by restriction, that is,
\begin{equation} \label{eq:fractionalSobolevHalfSpace}
    H^s_\alpha(\RR^3_-) := \Bigg\{ f\in \cD'(\overline{\RR^3_-});\, \exists g\in H^s_\alpha(\RR^3), \;
    f=g\lfloor_{\RR^3_-}\Bigg\}, 
\end{equation}
equipped with the norm:
\[
      \|f\|_{H^s_\alpha(\RR^3_-)} := \inf \{\|g\|_{H^s_\alpha};\, g\in H^s_\alpha(\RR^3), \, f=g\lfloor_{\RR^3_-} \}.
\]
Then, there exists a continuous restriction operator $R$ and extension operator $E$ such that  $R\,E$ is the 
identity map on $H^s_\alpha(\RR^3_-)$ (see e.g. \cite[Lemma 7]{TriebelWeight}, \cite[Theorem I.4]{Hanouzet}). 

We also define the spaces $\mathring{H}^s_\alpha(\RR^3_-)$, $0<s<2$, $0\leq \alpha$, by the analog of formula 
\eqref{eq: H m,alpha,0}, where there is a canonical extension operator given by the extension by zero to the whole $\RR^3$. Lastly, we denote by $H^{-s}_{-\alpha}(\RR^3_-)$ the dual of $\mathring{H}^{s}_{\alpha}(\RR^3_-)$.

Fractional spaces can also be obtained by means of real interpolation of integer-order spaces. This result will be 
needed in Section \ref{sec: the direct problem}. We will exploit the interpolation results in \cite{TriebelWeight}. (In that work, more general spaces $w^s_{\alpha,p}$ are studied; by Theorem 2 there,  the spaces $H^s_\alpha(\RR^3)$ agree with the space $w^s_{2,\alpha}(\RR^3)$ for any  $0\leq s\leq 2$ and any $\alpha \in \RR$).
 By formula (59) in \cite{TriebelWeight}, which is a particular case of Theorem 3 (a), using 
Formula (15) in that same work, the following result holds:
   \begin{multline}\label{eq:interp0}
         \qquad \qquad 
         \left[ H^{s_0}_{\alpha_0}(\RR^3),H^{s_1}_{\alpha_1}(\RR^3)\right]_{\Theta,2} = H^s_\alpha(\RR^3), \\
         s=s_0(1-\Theta) +s_1\,\Theta, \quad \alpha= \alpha_0(1-\Theta) +\alpha_1\,\Theta, \quad \Theta \in (0,1),
   \end{multline}       
for any $\alpha_0$, $ \alpha_1 \in \RR$, $s_0,\, s_1 \in [0,2]$, as long as $s_0\ne s_1$ and $s$ is not an integer.
The identification \eqref{eq:interp0} holds also for weighted spaces on $\RR^3_-$. In fact, by the definition of interpolation space (see e.g. \cite[Definition 2.4.1]{Bergh-Lofstrom}), the restriction and extension operators map between corresponding interpolation spaces, so:
\[
      \left[ R(H^{s_0}_{\alpha_0}(\RR^3)), R(H^{s_1}_{\alpha_1}(\RR^3))\right]_{\Theta,2} = R(H^s_\alpha(\RR^3)),
\]
but the restriction operator $R:H^s_\alpha(\RR^3)\to H^s_\alpha(\RR^3_-)$, $s\geq 0$, acts surjectively (see also Section 3.3 in \cite{TriebelWeight}), so that $ R(H^s_\alpha(\RR^3))= H^s_\alpha(\RR^3_-)$.
We specialize the interpolation formula to two cases of interest:
\begin{subequations} \label{eq:interp2} 
    \begin{equation}\label{eq:interp2.1} 
        \left[ H^{0}_{-1}(\RR^3_-),H^{1}_{0}(\RR^3_-)\right]_{\Theta,2} = H^\Theta_{-1+\Theta}(\RR^3_-),
    \end{equation}
     \begin{equation} \label{eq:interp2.2} 
       \left[ H^{1}_{0}(\RR^3_-),H^{2}_{1}(\RR^3_-)\right]_{\Theta,2} = H^{1+\Theta}_{\Theta}(\RR^3_-).
     \end{equation}
\end{subequations}

Next, we extend the interpolation analysis to negative spaces using duality. It is clear that \eqref{eq:interp2} holds for the spaces $\mathring{H}^s_\alpha(\RR^3_-)$, which are (not necessarily proper) closed subspaces of $H^s_\alpha(\RR^3_-)$. We then recall the following standard fact about real interpolation (see e.g. again \cite[Definition 2.4.1]{Bergh-Lofstrom}), which applies since the intersection of any of the spaces   
$\mathring{H}^s_\alpha(\RR^3_-)$ contains the dense subspace $\cD(\RR^3_-)$, specialized to the weighted spaces:
\[
      \left[ H^{-s_0}_{-\alpha_0}(\RR^3_-), H^{-s_1}_{-\alpha_1}(\RR^3_-)\right]_{\Theta,2} = 
       H^{-s}_{-\alpha}(\RR^3_-),
\]
using that $(\mathring{H}^s_\alpha)'(\RR^3_-) = H^{-s}_{-\alpha}(\RR^3_-)$. Finally, using that 
\[
     \left[ A_0,A_1\right]_{\Theta,2} = \left[ A_1,A_0\right]_{1-\Theta,2},
\]
(see e.g. \cite[Theorem 3.4.1]{Bergh-Lofstrom}) we conclude that:
\begin{subequations} \label{eq:interp3} 
    \begin{equation}\label{eq:interp3.1} 
        \left[ H^{-1}_{0}(\RR^3_-),H^{0}_{1}(\RR^3_-)\right]_{1-\Theta,2} = H^{-\Theta}_{1-\Theta}(\RR^3_-),
    \end{equation}
     \begin{equation} \label{eq:interp3.2} 
       \left[ H^{-2}_{-1}(\RR^3_-),H^{-1}_{0}(\RR^3_-)\right]_{1-\Theta,2} = H^{-1-\Theta}_{-\Theta}(\RR^3_-).
     \end{equation}
\end{subequations}
We close this section by recalling a trace result which we present only for spaces $H^m_{\alpha}(\mathbb{R}^3_-)$, where $m$ is a positive integer and $\alpha\in\mathbb{R}$.

\begin{proposition}[\cite{Amrouche-Dambrine-Raudin}, Lemma 1.1]\label{prop: trace theor}
Let $m\geq 1$ and $\alpha\in\mathbb{R}$. Then, there exists  a continuous linear mapping 
	\begin{equation}
		\bm{\gamma}=(\gamma_0,\cdots,\gamma_{m-1}): H^m_{\alpha}(\mathbb{R}^3_-)\longrightarrow 
		\prod_{j=0}^{m-1}H^{m-j-1/2}_{\alpha}(\mathbb{R}^2).
	\end{equation}
In addition,  $\bm{\gamma}$ is surjective and 
	\begin{equation*}
		\textup{\textrm{Ker}}\,\bm{\gamma}=\mathring{H}^{m}_{\alpha}(\mathbb{R}^3_-).
	\end{equation*}
\end{proposition}

	
\section{The direct problem: formulation and well-posedness}\label{sec: the direct problem}

In this section, we formulate the direct problem and study its well posedness. We begin by discussing in more detail 
the assumptions we make on the geometry of the dislocation surface $S$, on the displacement jump $\bg$ across $S$, and on the 
elasticity tensor $\CC$, which form the data for the direct problem.

\subsection{Main assumptions and a priori information} \label{sec:main_assumptions}

We recall that we model the dislocation surface  $S$ by  an open, bounded, oriented Lipschitz surface such that 
	\begin{equation}
		\overline{S}\subset \mathbb{R}^3_-.
	\end{equation}
We assume that  the interior of the Earth is an isotropic, inhomogeneous elastic medium. The associated elasticity tensor 
$\mathbb{C}=\mathbb{C}(\bm{x})$,  $\bm{x}\in\mathbb{R}^3_-$, is then of the form:
	\begin{equation}\label{ass: elast tensor}
		 \mathbb{C}(\bm{x})=\lambda(\bm{x})\mathbf{I}\otimes \mathbf{I}+2\mu(\bm{x})\mathbb{I},\qquad 
		 \forall\bm{x}\in\mathbb{R}^3_-,
	\end{equation}
where $\lambda=\lambda(\bm{x})$ and $\mu=\mu(\bm{x})$ are the Lam\'e coefficients. We suppose the Lam\'e 
parameters to have Lipschitz regularity, that is, there exists $M>0$ such that 
	\begin{equation}\label{ass: c1,1 regularity}
		\begin{aligned}
		\|\mu\|_{C^{0,1}(\overline{\mathbb{R}^3_-})}+\|\lambda\|_{C^{0,1}(\overline{\mathbb{R}^3_-})}\leq M,
		\end{aligned}
	\end{equation} 
with $\|\cdot\|_{C^{0,1}(\overline{\mathbb{R}^3_-)}}=\|\cdot\|_{L^{\infty}(\mathbb{R}^3_-)}+\|\nabla\cdot\|_{L^{\infty}
(\mathbb{R}^3_-)}$.
Moreover, we require that there exists a constant $C>0$ such that
	\begin{equation}\label{eq: decay_cond_lame}
		|\nabla \lambda|\leq \frac{C}{\varrho},\qquad  |\nabla \mu|\leq \frac{C}{\varrho},\qquad \text{a.e. in}\,
		\overline{\mathbb{R}^3_-},
	\end{equation} 
where $\varrho$ is the weight defined in \eqref{eq: weight}.	
Finally, we assume the strong convexity condition for the elasticity tensor, i.e., there exist two positive constants 
$\alpha_0, \beta_0$ such that
	\begin{equation}\label{ass: strong conv}
		\mu(\bm{x})\geq \alpha_0>0,\qquad\qquad 3\lambda(\bm{x})+2\mu(\bm{x})\geq \beta_0>0,\quad \forall 
		\bm{x}\in \mathbb{R}^3_-.
	\end{equation}
This condition implies that $\CC$ defines a positive-definite quadratic form on symmetric matrices: {For all $\widehat{\mathbf{A}}\in\mathbb{R}^{3x3}$ we have} 
	\begin{equation}\label{eq: strong convexity}
		\mathbb{C}(\bm{x})\widehat{\mathbf{A}}:\widehat{\mathbf{A}}\geq c |\widehat{\mathbf{A}}|, \qquad 
		\forall\bm{x}\in\mathbb{R}^3_-,
	\end{equation}
for $c=c(\alpha_0,\beta_0)>0$.

\subsection{The transmission problem as a source problem}

The presence of a non-trivial jump $\bg$ for the displacement $\bu$ across the dislocation leads to reformulate the 
transmission-boundary-value problem  \eqref{eq: transm_problem_intro} as boundary-value problem with rough 
interior source. In this way, we are able to prove an existence and uniqueness result for the solution globally in 
$\mathbb{R}^3_-$. This approach has been proposed in \cite{ColliFranzone-Guerri-Magenes} for a similar 
transmission problem, but in the case of {the Laplace operator}. 

We recall Problem  \eqref{eq: transm_problem_intro}  here for the reader's sake:
		\begin{equation}\label{eq: Pu}
			\begin{cases}
				\textrm{\textup{div}}\, (\mathbb{C}\widehat{\nabla}\bm{u})=\bm{0},& \textrm{in}\,\, 
				\mathbb{R}^3_- \setminus \overline{S},\\
				(\mathbb{C}\widehat{\nabla}\bm{u})\bm{e}_3=\bm{0}, & \textrm{on}\,\, \{x_3=0\}, \\
				[\bm{u}]_{S}=\bm{g}, \\
				[(\mathbb{C}\widehat{\nabla}\bm{u})\bm{n}]_{S}=\bm{0} ,\\
			\end{cases}
		\end{equation}
where $\bm{e}_3=(0,0,1)$ is the  outward unit normal vector on $\{x_3=0\}$, $\bm{n}$ is a unit normal vector on $S$, and 
$\bm{g}$ is a vector field on $S$ such that
\begin{equation}
		\bm{g}\in H^{1/2}(S).
\end{equation} 		
We rewrite this problem in the form
	\begin{equation}\label{le: distr_sol}
		\begin{cases}
				\textrm{\textup{div}}\, (\mathbb{C}\widehat{\nabla}\bm{u})=\bm{f}_S, & \textrm{in}\,\, 
				\mathbb{R}^3_-,\\
				(\mathbb{C}\widehat{\nabla}\bm{u})\bm{e}_3=\bm{0}, & \textrm{on}\,\, \{x_3=0\},\\
		\end{cases}
	\end{equation} 	
where
	\begin{equation}\label{eq: source_term}
			\bm{f}_{S}=\textrm{div}(\mathbb{C}(\bm{g}\otimes  \bm{n})\delta_{S}).
	\end{equation}
Above, $\delta_{S}$ is the distribution on $\RR^3$ defined by:
\begin{equation}\label{eq: def_delta}
     	\langle\delta_S, \bm{\phi}\rangle = \int_S \bm{\phi}(\bm{x})\, d\sigma(\bm{x}), \qquad \forall \bm{\phi}\in D(\RR^3),
\end{equation}
and, if $h\in L^1(S)$, the distribution $h\, \delta_S$ is defined by 
\begin{equation}\label{eq: def_hdelta}
  \langle h\, \delta_S, \bm{\phi}\rangle = \int_S h(\bm{x})\,\bm{\phi}(\bm{x})\, d\sigma(\bm{x}), \qquad \forall \bm{\phi}\in D(\RR^3),
\end{equation}
We  show the equivalence between \eqref{eq: Pu} and \eqref{le: distr_sol} later in Lemma \ref{lem: equivalence_two_problems}. Next, we establish the validity of Formula \eqref{eq: source_term} and investigate the regularity of the source term.

\begin{proposition}\label{prop: sobolev_space_source_term}
The source term $\bm{f}_S\in H^{-3/2-\varepsilon}(\mathbb{R}^3_-)$. 
\end{proposition}	

\begin{proof}
It is sufficient to show that  $\mathbb{C}(\bm{g}\otimes\bm{n})\delta_{S}$ is a distribution in the space 
$ H^{-1/2-\varepsilon}(\mathbb{R}^3_-)$, for all $\varepsilon>0$. To this end, we introduce the function 
$\mathbf{\Psi}:=\mathbb{C}(\bm{g}\otimes\bm{n})$, and observe that it lies  in $L^2(S)$, since $\bm{n}\in 
L^\infty(S)$, $\mathbb{C}$ is bounded and continuous on $S$ by Assumption \eqref{ass: c1,1 regularity},  and 
$\bm{g}$ is assumed to be in  $H^{1/2}(S)$. Next, we fix a bounded Lipschitz domain $\Omega$ such that $\overline{\Omega}\subset 
\mathbb{R}^3_-$ and $S\subset\partial\Omega$. The existence of this domain is guaranteed by the assumed regularity on $S$ and the hypothesis that $S$ is at positive distance from $\{x_3=0\}$. For all $\bm{\varphi}\in H^{1/2+\varepsilon}
(\mathbb{R}^3_-)$, we have 
	\begin{equation*}
		\begin{aligned}
			\Bigg|\int\limits_{S}\mathbf{\Psi}(\bm{x})\bm{\varphi}(\bm{x})\, d\sigma(\bm{x})\Bigg|\leq  \|
			\mathbf{\Psi}\|_{L^2(S)}\|\bm{\varphi}\|_{L^2(S)}&\leq c \, \|\bm{\varphi}\|
			_{L^2(\partial\Omega)}\\
			&\leq c\|\bm{\varphi}\|_{H^{\varepsilon}(\partial\Omega)}
			\leq c \|\bm{\varphi}\|_{H^{1/2+\varepsilon}(\mathbb{R}^3_-)},	
		\end{aligned}
	\end{equation*}
where  the fourth inequality follows by the Trace Theorem. Therefore, by definition  $\mathbf{\Psi}\delta_{S}\in 
H^{-1/2-\varepsilon}(\mathbb{R}^3_-)$, which implies that $\bm{f}_S \in H^{-3/2-\varepsilon}(\RR^3_-)$.
\end{proof}

The well-posedness of Problem \eqref{le: distr_sol}, with source term given by \eqref{eq: source_term}, will be
obtained  by duality and interpolation, once we have regularity for the source problem with source term in spaces 
of positive  regularity. To do so, we study the general source problem:
\begin{equation}\label{eq: reg_prob}
	\begin{cases}
		\textrm{\textup{div}}\, (\mathbb{C}\widehat{\nabla}\bm{u})=\bm{f}, & \textrm{in}\,\, \mathbb{R}^3_-, \\
		(\mathbb{C}\widehat{\nabla}\bm{u})\bm{e}_3=\bm{0}, & \textrm{on}\,\, \{x_3=0\}.\\
	\end{cases}
\end{equation} 	
We follow the well-known approach of Lions and Magenes  \cite{Lions-Magenes} to establish regularity for this problem. This approach was already adapted to the framework of weighted spaces in a half space in  \cite{Hanouzet} and  \cite{Amrouche-Dambrine-Raudin} for the Dirichlet and Neumann problems. However, our source problem is more  singular, and we assume lower regularity on the coefficients.

Our strategy involves strong, weak, and very weak solutions, duality and interpolation. 
To introduce the proper functional setting,  we need to define some auxiliary spaces.

\begin{definition}
Let $V(\mathbb{R}^3_-)$ be the closed subspace of $H^2_1(\mathbb{R}^3_-)$ given by
		\begin{equation}\label{eq: space V}
			V(\mathbb{R}^3_-)=\Big\{\bm{v}\in H^2_1(\mathbb{R}^3_-)\,\, / \,\, 
			(\mathbb{C}\widehat{\nabla}\bm{v})\bm{e}_3=\bm{0}\,\, \textrm{on}\,\, \{x_3=0\}  \Big\},
		\end{equation}
with norm $\|\bm{v}\|_{V(\mathbb{R}^3_-)}=\|\bm{v}\|_{H^2_1(\mathbb{R}^3_-)}$.		
We denote the dual space of $V(\mathbb{R}^3_-)$ with $V'(\mathbb{R}^3_-)$.
\end{definition} 

\begin{remark}
We note that if $\bm{v}\in V(\mathbb{R}^3_-)$ then in particular $\bm{v}\in H^0_{-1}(\mathbb{R}^3_-)$ which 
implies that no infinitesimal rigid motion $\bm{v}=\mathbf{A}\bm{x}+\bm{c}$, where 
$\mathbf{A}\in\mathbb{R}^{3\times 3}$ is a skew matrix and $\bm{c}\in\mathbb{R}^3$, can be an element of 
$V(\mathbb{R}^3_-)$. 
\end{remark}

\begin{definition}
Let
	\begin{equation}\label{eq: space E^2_0}
	E_0(\mathbb{R}^3_-)=\Big\{\bm{u}\in H^0_{-1}(\mathbb{R}^3_-)\,\, /\,\, \textup{\textrm{div}}
	(\mathbb{C}\widehat{\nabla}\bm{u}) \in V'(\mathbb{R}^3_-) \Big\},
	\end{equation}
equipped with the norm
	\begin{equation*}
		\|\bm{u}\|_{E_0(\mathbb{R}^3_-)}=\|\bm{u}\|_{H^0_{-1}(\mathbb{R}^3_-)}+ \|\textup{\text{div}}
		(\mathbb{C}\widehat{\nabla}\bm{u})\|_{V'(\mathbb{R}^3_-)}.
	\end{equation*}		
\end{definition}

For the reader's convenience we give here a sketch of the strategy we use to establish well-posedness of \eqref{eq: 
reg_prob}:
\begin{enumerate}[label=(\roman*), ref=(\roman*)]
\item we prove that, for any $\bm{f}\in (H^1_0)'(\mathbb{R}^3_-)$, there exists a unique weak solution $\bm{u}\in 
H^1_0(\mathbb{R}^3_-)$;  \label{i.1}
\item if $\bm{f}\in H^0_1(\mathbb{R}^3_-)$, we show that the weak solution in \ref{i.1} is, in fact, a strong solution 
$\bm{u}\in H^2_1(\mathbb{R}^3_-)$; \label{i.2} 
\item from \ref{i.2}  by a duality argument, we establish that, for any $\bm{f}\in V'(\mathbb{R}^3_-)$, there exists a 
unique solution $\bm{u}\in H^0_{-1}(\mathbb{R}^3_-)$, that is, $\bm{u}\in E_0(\mathbb{R}^3_-)$; \label{i.3}
\item the well-posedness of \eqref{le: distr_sol}  with source term \eqref{eq: source_term} follows by interpolating the  results  obtained in \ref{i.1} and \ref{i.3}. \label{i.4}   
\end{enumerate}

\subsection{Well-posedness}

In this section we prove the well-posedness of Problem \eqref{le: distr_sol} with $\bm{f}_S$ given by \eqref{eq: 
source_term}. 
Following the strategy presented in the scheme of the previous section, we start by proving \ref{i.1}.

\begin{theorem}[Weak solution]\label{lem: weak_sol}
Under Assumptions \eqref{ass: c1,1 regularity}, \eqref{eq: strong convexity}, for any $\bm{f}\in \
(H^1_0)'(\mathbb{R}^3_-)$ there exists a unique solution $\bm{u}\in H^1_0(\mathbb{R}^3_-)$ of Problem 
\eqref{eq: reg_prob} such that
	\begin{equation}\label{eq: estimate solution in H^1_0}
		\|\bm{u}\|_{H^1_0(\mathbb{R}^3_-)}\leq C \|\bm{f}\|_{(H^1_0)'(\mathbb{R}^3_-)}.
	\end{equation} 
\end{theorem}

\begin{remark}
Above and in Theorems \ref{th: sol in H^2_1}, \ref{th: well-posed in H^0_-1}, and \ref{th: very_weak_sol}, the 
constants generally depend on the constants in Assumptions \eqref{ass: c1,1 regularity}, 
\eqref{eq: decay_cond_lame}, and  \eqref{ass: strong conv},  such as $M$ and $c$, but are independent of $S$ and 
$\bg$.\\ 
\end{remark}

\begin{proof}
We introduce the bilinear form $a: H^1_0(\mathbb{R}^3_-)\times H^1_0(\mathbb{R}^3_-)\to \mathbb{R}$ associated to the Lam\'e  operator:
	\begin{equation}\label{eq: bil_form_a}
		a({\bm{w}},\bm{v}) := \int_{\mathbb{R}^3_-}\mathbb{C}\widehat{\nabla}{\bm{w}}: \widehat{\nabla}\bm{v}\, 
		d\bm{x}.
	\end{equation}
Since $\mathcal{D}(\overline{\mathbb{R}^3_-})$ is dense in $H^1_0(\mathbb{R}^3_-)$, we can assume in the 
calculations below that  $\bm{w}\in\mathcal{D}(\overline{\mathbb{R}^3_-})$.   We test the first equation in \eqref{eq: reg_prob} with  $\bm{v}\in  H^1_0(\mathbb{R}^3_-)$, {and temporarily substitute $\bm{w}$ for $\bm{u}$}, we then  integrate by parts once {with respect to} ${\bm{w}}$ over $\mathbb{R}^3_-$ on the left-hand side, using the  traction-free boundary condition, to obtain:
{\begin{equation*}
		\int_{\mathbb{R}^3_-}\mathbb{C}\widehat{\nabla}\bm{w}: \widehat{\nabla}\bm{v}\, d\bm{x}= -
		\langle  \bm{f}, \bm{v}\rangle_{((H^1_0)'(\RR^3_-),H^1_0(\RR^3_-))}, \qquad \forall  \bm{v}\in  H^1_0(\mathbb{R}^3_-).
	\end{equation*}
For the density result of $\mathcal{D}(\overline{\mathbb{R}^3_-})$ in $H^1_0(\mathbb{R}^3_-)$, the previous expression also holds for $\bm{u}\in H^1_0(\mathbb{R}^3_-)$, hence the weak formulation associated to \eqref{eq: reg_prob} is}	
	\begin{equation*}
          \int_{\mathbb{R}^3_-}\mathbb{C}\widehat{\nabla}\bm{u}: \widehat{\nabla}\bm{v}\, d\bm{x}= -
         \langle  \bm{f}, \bm{v}\rangle_{((H^1_0)'(\RR^3_-),H^1_0(\RR^3_-))}, \qquad \forall  \bm{v}\in  H^1_0(\mathbb{R}^3_-).
	\end{equation*}
We have hence the following  variational formulation of Problem \eqref{eq: reg_prob}:\\
\textit{find $\bm{u}\in H^1_0(\mathbb{R}^3_-)$ such that  
	\begin{equation*}
		a(\bm{u},\bm{v})=l(\bm{v}),\quad \forall\bm{v}\in H^1_0(\mathbb{R}^3_-),
	\end{equation*}
where $l: H^1_0(\mathbb{R}^3_-)\to \mathbb{R}$ is the  linear functional given by}
	\begin{equation}\label{eq: func}
		l(\bm{v})= - \langle  \bm{f}, \bm{v}\rangle_{((H^1_0)'(\RR^3_-),H^1_0(\RR^3_-))}.
	\end{equation}
The assertion of the theorem then follows by applying the Lax-Milgram theorem, once the continuity and coercivity of the bilinear form $a$ and the continuity of the functional $l$ are established. {We stress that, due to the weighted Poincar\'e and Korn inequalies 
\eqref{eq: Poincare}, the coercivity of the bilinear form $a(\bm{u},\bm{v})$ with respect to the inner product in the Hilbert space $H^1_0(\RR^3_-)$ is equivalent to the usual coercivity of $a$ with respect to the standard $H^1$-seminorm, if $\CC$ is strongly convex.}

\noindent\textit{Continuity and coercivity of \eqref{eq: bil_form_a}}:\  From \eqref{ass: c1,1 regularity}, we have
	\begin{equation*}
		\begin{aligned}
		|a(\bm{u},\bm{v})|=\Bigg|\int_{\mathbb{R}^3_-}\mathbb{C}\widehat{\nabla}\bm{u}:
		\widehat{\nabla}\bm{v}\, d\bm{x}\Bigg|&\leq c \|\widehat{\nabla}\bm{u}\|_{L^2(\mathbb{R}^3_-)}\|
		\widehat{\nabla}\bm{v}\|_{L^2(\mathbb{R}^3_-)}\\
		&\leq c \|\bm{u}\|_{H^1_0(\mathbb{R}^3_-)}\|\bm{v}\|_{H^1_0(\mathbb{R}^3_-)}.
		\end{aligned}
	\end{equation*}	
Coercivity follows from the strong convexity of $\mathbb{C}$ (Assumptions \eqref{ass: strong conv} and \eqref{eq: 
strong convexity}), by applying the weighted Korn's inequality \eqref{eq: Korn}:
	\begin{equation*}
		\begin{aligned}
			a(\bm{u},\bm{u})=\int_{\mathbb{R}^3_-}\mathbb{C}\widehat{\nabla}\bm{u}:
			\widehat{\nabla}\bm{u}\, d\bm{x}&\geq c \|\widehat{\nabla}\bm{u}\|^2_{L^2(\mathbb{R}^3_-)}\\
			&\geq c \|{\nabla}\bm{u}\|^2_{L^2(\mathbb{R}^3_-)}\geq c\|\bm{u}\|^2_{H^1_0(\mathbb{R}^3_-)}.
		\end{aligned}
	\end{equation*} 
\textit{Continuity of \eqref{eq: func}:}\ $l$ is by definition an element of the dual of $H^1_0(\RR^3_-)$, hence it defines a continuous linear functional on $H^1_0$ via the duality form.

The  conclusion now follows from the Lax-Milgram Theorem.		
\end{proof}

Next  we establish \ref{i.2}, i.e., the fact that  the unique weak solution is actually a strong solution, if the source 
term in \eqref{eq: reg_prob} belongs to $H^0_1(\mathbb{R}^3_-)$.

\begin{theorem}[Strong solution]\label{th: sol in H^2_1}
Under the assumptions of Theorem \ref{lem: weak_sol} and Assumption \eqref{eq: decay_cond_lame} on the 
decay of the Lam\'e coefficients, if the source $\bm{f} \in  H^0_1(\mathbb{R}^3_-)$,  Problem \eqref{eq: reg_prob} 
admits a unique solution $\bm{u}\in H^2_1(\mathbb{R}^3_-)$, satisfying
		\begin{equation}\label{eq: stab_est_H^2_1}
			\|\bm{u}\|_{H^2_1(\mathbb{R}^3_-)}\leq c \|\bm{f}\|_{H^0_1(\mathbb{R}^3_-)}.
		\end{equation}	
\end{theorem}  
\begin{proof}
We begin by observing that, if $\bm{f}\in H^0_1(\RR^3_-)$, then $\bm{f}\in (H^1_0)'(\RR^3_-)$. In fact,  using that $\varrho(x)\geq 1$ and that $\bm{f}$ is now a locally integrable function, we can estimate \eqref{eq: func} as follows:
\begin{equation}\label{eq: finH0_1}
|l(\bm{v})|=\Bigg|\int_{\mathbb{R}^3_-} \bm{f}\cdot \bm{v}\, d\bm{x}\Bigg|\leq \|\varrho \bm{f}\|
_{L^2(\mathbb{R}^3_-)}\Big\|\frac{\bm{v}}{\varrho}\Big\|_{L^2(\mathbb{R}^3_-)}\leq \|\bm{f}\|
_{H^0_1(\mathbb{R}^3_-)}\|\bm{v}\|_{H^1_0(\mathbb{R}^3_-)}. 
\end{equation} 
By uniqueness of weak solutions, it is enough to show that we can bootstrap regularity for the source problem, if 
$\bm{f}$ is more regular. To bootstrap, we will prove that suitable weighted derivatives of the solutions satisfy a 
similar source problem. To be specific,  we study the source-boundary-value  problem formally satisfied by 
$\varrho\partial_i \bm{u}$, for $i=1,2$. Since we take only derivatives that are tangent to the boundary, we can 
show that this problem is in a form similar to  the original problem \eqref{eq: reg_prob}  (see 
\eqref{eq: bvp_rho_partial_u_new}). Theorem \ref{lem: weak_sol} then gives, again by uniqueness of solutions, 
that $\varrho\partial_i \bm{u}\in H^1_0(\mathbb{R}^3_-)$, which in turn means  by  \eqref{map: isomorph} that $\partial_i\bm{u}\in  H^1_1(\mathbb{R}^3_-)$ or, equivalently,  that $\varrho\nabla(\partial_i \bm{u})\in L^2(\mathbb{R}^3_-)$ . Lastly, by using the specific form of the Lam\'e system, we are able to prove that the regularity of the tangential  derivatives implies  $\varrho\partial^2_3\bm{u}\in L^2(\mathbb{R}^3_-)$ as well.
Therefore, we can conclude that  $\varrho\partial^2 \bm{u}\in L^2(\mathbb{R}^3_-)$.

\subsubsection*{First Step:}\ 
We seek to find the problem satisfied by $\varrho \partial_i \bm{u}$, for $i=1,2$, knowing that $\bm{u}$ satisfies 
\eqref{eq: reg_prob}. We proceed formally first. The manipulations below are justified {\em a posteriori}, given the regularity on $\bu$, $\varrho$, and the data.
We have that
	\begin{equation*}
			\textrm{div}(\mathbb{C}\widehat{\nabla}(\varrho \partial_i\bm{u}))=\varrho\partial_i\bm{f} + 
			\textrm{div}
			(\mathbb{C}\widehat{(\nabla\varrho\otimes\partial_i\bm{u})})+(\mathbb{C}\widehat{\nabla}
			(\partial_i\bm{u}))\nabla\varrho-\varrho\,\textrm{div}\,(\partial_i\mathbb{C}\widehat{\nabla}\bm{u}),
	\end{equation*}
and, since 
	\begin{equation*}
		\varrho\,\textrm{div}(\partial_i\mathbb{C}\widehat{\nabla}\bm{u})=\textrm{div}
                 (\varrho\partial_i\mathbb{C}\widehat{\nabla}\bm{u})-\partial_i\mathbb{C}\widehat{\nabla}\bm{u}\nabla
                 \varrho,
	\end{equation*}	
we find that 
	\begin{equation}\label{eq: equation_solved_rhopartial_i}
	\textrm{div}(\mathbb{C}\widehat{\nabla}(\varrho \partial_i\bm{u})-{\bf{G}}_i)= 
	\bm{F}_i,    
	\qquad 	\textrm{in}\,\mathbb{R}^3_-,
	\end{equation}
where 
	\begin{equation*}
		\bm{F}_i:=\varrho\partial_i\bm{f} +(\mathbb{C}\widehat{\nabla}
		(\partial_i\bm{u}))\nabla\varrho+\partial_i\mathbb{C}\widehat{\nabla}\bm{u}\nabla\varrho,\,\quad 
		\textrm{and}\quad {\bf{G}}_i:=\mathbb{C}\widehat{(\nabla\varrho\otimes\partial_i\bm{u})} - 
		\varrho\partial_i\mathbb{C}\widehat{\nabla}\bm{u}.
	\end{equation*}	
Next, from the Neumann boundary condition on $\bu$ in \eqref{eq: reg_prob}, it follows that	
\begin{equation*}
		\mathbb{C}\widehat{\nabla}
                (\varrho'\partial_i\bm{u})\bm{e}_3=\mathbb{C}\widehat{(\nabla\varrho'\otimes\partial_i\bm{u})}
                \bm{e}_3-\varrho'(\partial_i\mathbb{C}\widehat{\nabla}\bm{u})\bm{e}_3, \qquad \text{ for } x_3=0, 
	\end{equation*}
where $\varrho'=\varrho\lfloor_{\mathbb{R}^2}$. That is, 
	\begin{equation}\label{eq: bc_rhopartial_i}
		(\mathbb{C}\widehat{\nabla}(\varrho'\partial_i\bm{u}) -  {\bf{G}}_i) \bm{e}_3= \bm{0},\qquad \textrm{on}\, 
		\{x_3=0\}.
	\end{equation}
Hence, by combining \eqref{eq: equation_solved_rhopartial_i} and \eqref{eq: bc_rhopartial_i}, we obtain the problem
\begin{equation}\label{eq: bvp_rho_partial_u_new}
	\begin{cases}
		\textup{\textrm{div}}(\mathbb{C}\widehat{\nabla}(\varrho\partial_i\bm{u}) - {\bf{G}}_i)=\bm{F}_i,  & \textrm{in}\, \mathbb{R}^3_-,\\
		(\mathbb{C}\widehat{\nabla}(\varrho'\partial_i\bm{u})- {\bf{G}}_i) \bm{e}_3 = \bm{0}, & \textrm{on}\, 
		 \{x_3=0\}.
	\end{cases}
\end{equation}
We write the weak formulation of this problem.  The quadratic form associated to the left-hand side of the equation above is: 
\begin{equation} \label{eq: Gintegral} 
    \int_{\RR^3_-}    (\mathbb{C}\widehat{\nabla}(\varrho\partial_i\bm{u})-{\bf{G}}_i) \cdot \nabla \bm{\varphi}\, d \bx,
\end{equation}
for $\bm{\varphi}\in H^1_0(\RR^3_-)$. This expression is justified by the fact that $\bu\in H^1_0(\RR^3_-)$ and ${\bf{G}}_i\in L^2(\RR^3_-)$ by the regularity and decay conditions on $\CC$ (in particular, the fact that $\varrho \nabla \CC\in L^\infty(\RR^3_-)$ from \eqref{eq: decay_cond_lame}).

We further observe that $\bm{F}_i\in (H^1_0)'(\mathbb{R}^3_-)$. In fact, taking again $\bm{\varphi}\in {H}^1_{0}
(\mathbb{R}^3_-)$, by the hypothesis $\bm{f}\in H^0_1(\mathbb{R}^3_-)$, Theorem \ref{lem: weak_sol} and \eqref{eq: finH0_1}, we have
	\begin{equation}\label{eq: F5}
		\begin{aligned}
			\Big|\langle \varrho\partial_i \bm{f},\bm{\varphi}\rangle_{((H^{1}_0)'(\mathbb{R}^3_-),{H}^1_{0}
			(\mathbb{R}^3_-))}\Big|&=\Big|\langle \partial_i(\varrho \bm{f})-\bm{f}\partial_i\varrho ,
			\bm{\varphi}\rangle_{((H^{1}_0)'(\mathbb{R}^3_-),{H}^1_{0}(\mathbb{R}^3_-))}\Big|\\
			&\leq C\left( \|\bm{f}\|_{H^0_1(\mathbb{R}^3_-)}\|\nabla\bm{\varphi}\|_{L^2(\mathbb{R}^3_-)}+\|
			\bm{f}\|_{H^0_1(\mathbb{R}^3_-)}\bigg\|\frac{\bm{\varphi}}{\varrho}
			\bigg\|_{L^2(\mathbb{R}^3_-)}   
			\right)\\
			&\leq C \ \|\bm{f}\|_{H^0_1(\mathbb{R}^3_-)} \|\bm{\varphi}\|_{{H^1_{0}}(\mathbb{R}^3_-)}.
		\end{aligned}
	\end{equation}
 We point out that, in the first inequality above,  there are no boundary terms, because we take tangential  derivatives.  
Next, from \eqref{eq: decay_cond_lame}, Theorem \ref{lem: weak_sol} and \eqref{eq: finH0_1} it follows that
	\begin{equation}\label{eq: F2}
		\begin{aligned}
		\Big|\langle (\mathbb{C}\widehat{\nabla}\partial_i\bm{u})\nabla\varrho,
		\bm{\varphi}\rangle_{((H^{1}_0)'(\mathbb{R}^3_-),{H}^1_{0}(\mathbb{R}^3_-))}\Big|&=\Big|\langle 
		\widehat{\nabla}\bm{u},\partial_i(\mathbb{C}(\nabla\varrho\otimes\bm{\varphi}))\rangle\Big|\\
		&\qquad \qquad \leq C\ \|\widehat{\nabla}\bm{u}\|_{L^2(\mathbb{R}^3_-)}\|\bm{\varphi}\|_{{H^1_{0}}
		(\mathbb{R}^3_-)}\\
		&{\qquad\qquad \leq C\ \|\bm{f}\|_{H^0_1(\mathbb{R}^3_-)} \|\bm{\varphi}\|
		_{{H^1_{0}}(\mathbb{R}^3_-)}}.
		\end{aligned}
	\end{equation}
where the {first} inequality comes from the fact that $\varrho \,\partial_i\nabla\varrho$ are bounded in 
$\mathbb{R}^3_-$, for $i=1,2$. Again we stress that there are not boundary elements because we take tangential 
derivatives. 
Finally, using again the fact that $\nabla\varrho$ and $\varrho \,\partial_i\mathbb{C}$ are bounded, and Theorem \ref{lem: weak_sol} and \eqref{eq: finH0_1}, we get
 	\begin{equation}\label{eq: F4}
 		\begin{aligned}
	 		\Big|\langle (\nabla\varrho^T\partial_i\mathbb{C}\widehat{\nabla}\bm{u}) , \bm{\varphi}
	 		\rangle_{((H^{1}_0)'(\mathbb{R}^3_-),{H}^1_{0}(\mathbb{R}^3_-))}  \Big|
 			&\leq C \|\widehat{\nabla}\bm{u}\|_{L^2(\mathbb{R}^3_-)}\bigg\|\frac{\bm{\varphi}}{\varrho}\bigg\|
 			_{L^2(\mathbb{R}^3_-)} \\
 			&\qquad \qquad \leq C\ \|\widehat{\nabla}\bm{u}\|_{L^2(\mathbb{R}^3_-)} \|\bm{\varphi}\|
 			_{{H^1_{0}}(\mathbb{R}^3_-)}\\
 			&{\qquad\qquad \leq C\ \|\bm{f}\|_{H^0_1(\mathbb{R}^3_-)} \|\bm{\varphi}\|
 			_{{H^1_{0}}(\mathbb{R}^3_-)}}.
 		\end{aligned}
 	\end{equation}
By \eqref{eq: Gintegral} then, we can write Problem \eqref{eq: bvp_rho_partial_u_new} in weak form as:
\begin{equation} \label{eq: bilinear_form_a(w,phi)=h(phi)}
       \int_{\RR^3_-}
      (\mathbb{C}\widehat{\nabla}(\varrho\partial_i\bm{u}))\cdot \nabla \bm{\varphi}\, d\bx=  \int_{\RR^3_-} {\bf{G}}_i\cdot \nabla \bm{\varphi}\, d\bx
     - \langle \bm{F}_i, 
     \bm{\varphi}\rangle_{((H^1_0)'(\RR^3_-),H^1_0(\RR^3_-))}, 
\end{equation}
for  $\bm{\varphi}\in H^1_0(\RR^3_-)$.     
Since ${\bf{G}}_i\in L^2(\RR^3_-)$, it is clear that the right-hand side of the equality above defines a continuous functional, which we call $h_i$, $i=1,2$, on $H^1_0(\RR^3_-)$. In fact, we have already explicitly estimated the terms containing ${\bm{F}}_i$. For the term containing ${\bf{G}}_i$, we note that, using \eqref{ass: c1,1 regularity}, Theorem \ref{lem: weak_sol} and \eqref{eq: finH0_1},  we obtain
	\begin{equation}\label{eq: F6}
	 \begin{aligned}
		\Bigg| \int_{\mathbb{R}^3_-}\mathbb{C}\widehat{(\nabla\varrho \otimes \partial_i\bm{u} )}: 
		\widehat{\nabla}\bm{\varphi}\, d\bm{x} \Bigg| &\leq C \| \partial_i \bm{u}\|_{L^2(\mathbb{R}^3_-)} \|
		\nabla \bm{\varphi}\|_{L^2(\mathbb{R}^3_-)} \\ 
		&\qquad \qquad \leq C \| \partial_i \bm{u}\|_{L^2(\mathbb{R}^3_-)} \|
		\bm{\varphi}\|_{H^1_0(\mathbb{R}^3_-)},\\
		&{\qquad\qquad \leq C\ \|\bm{f}\|_{H^0_1(\mathbb{R}^3_-)} \|\bm{\varphi}\|
			_{{H^1_{0}}(\mathbb{R}^3_-)}}
         \end{aligned}
	\end{equation}
 and by \eqref{eq: decay_cond_lame},
	\begin{equation}\label{eq: F7}
		\begin{aligned}
		\Bigg| \int_{\mathbb{R}^3_-}\varrho \partial_i\mathbb{C}\widehat{\nabla}\bm{u}: \nabla \bm{\varphi}\,  
		d\bm{x} \Bigg| \leq C \| \nabla \bm{u}\|_{L^2(\mathbb{R}^3_-)} 
		\|\nabla \bm{\varphi}\|_{L^2(\mathbb{R}^3_-)} &\leq C \| \nabla \bm{u}\|_{L^2(\mathbb{R}^3_-)} \|
		\bm{\varphi}\|_{H^1_0(\mathbb{R}^3_-)}\\
		&{\leq C\ \|\bm{f}\|_{H^0_1(\mathbb{R}^3_-)} \|\bm{\varphi}\|
			_{{H^1_{0}}(\mathbb{R}^3_-)}}
		\end{aligned}
	\end{equation}	
We therefore study the variational problem: \\
{\em find $\bm{w}_i\in H^1_0(\mathbb{R}^3_-)$, for $i=1,2,$ such that 
	\begin{equation}
			a(\bm{w}_i,\bm{\varphi})=h_i(\bm{\varphi}), \qquad \forall \bm{\varphi}\in H^1_0(\mathbb{R}^3_-),
	\end{equation}	
where  $a$ is the bilinear form defined in \eqref{eq: bil_form_a}, and $h_i$, $i=1,2$, is the linear operator on $H^1_0(\RR^3_-)$ defined by the right-hand side in \eqref{eq: bilinear_form_a(w,phi)=h(phi)} as a function of $\bm{\varphi}$.}\\
If this problem has a unique solution, then necessarily $\bm{w}_i= \varrho \partial_i \bm{u}$, $i=1,2$.

As proved in Theorem \ref{lem: weak_sol}, the bilinear form $a$ is continuous and coercive, and we have shown above $h_i$ is continuous. Hence, the  Lax-Milgram Theorem then ensures the existence of a unique solution  $\bm{w}_i= \varrho \partial_i \bm{u}$ of 
\eqref{eq: bvp_rho_partial_u_new} in $H^1_0(\mathbb{R}^3_-)$, satisfying
	\begin{equation*}
		\|\varrho\partial_i\bm{u}\|_{H^1_0(\mathbb{R}^3_-)}\leq C \|\bm{f}\|_{H^{0}_1(\mathbb{R}^3_-)},
	\end{equation*}
{where the last inequality comes from \eqref{eq: F5}, \eqref{eq: F2}, \eqref{eq: F4}, \eqref{eq: F6} and \eqref{eq: F7}}.	
	
\subsubsection*{Second Step:}\ 
From the previous step, by the isomorphism \eqref{map: isomorph} we have that, for $i=1,2$,
$\partial_i\bm{u}\in H^1_1(\mathbb{R}^3_-)$ and satisfies  the estimate:
	\begin{equation}\label{eq: estimate for H^1_1 with f}
		\|\partial_i \bm{u}\|_{H^1_1(\mathbb{R}^3_-)}\leq C\, \|\varrho\partial_i\bm{u}\|
		_{H^1_0(\mathbb{R}^3_-)}\leq C\|\bm{f}\|_{H^0_1(\mathbb{R}^3_-)}.
	\end{equation}
If  $\partial_3\bm{u}\in H^1_1(\mathbb{R}^3_-)$, then
$\nabla\bm{u}\in H^1_1(\mathbb{R}^3_-)$, that is, $\bm{u}\in H^2_1(\mathbb{R}^3_-)$, given that $\bu\in 
H^1_0(\RR^3_-)$ by Theorem \ref{lem: weak_sol}. In fact,  it is enough to 
prove  that $\varrho\partial^2_3\bm{u}\in L^2(\mathbb{R}^3_-)$. To this end,  we use equation 
\eqref{eq: reg_prob},  written in the form:
	\begin{equation}\label{eq: explicit_equ_pb}
		\mu \Delta \bm{u} + (\lambda+\mu) \nabla\textrm{div}\,\bm{u}+\nabla\lambda\, \textrm{div}\,
		\bm{u}+2\widehat{\nabla}\bm{u}\nabla\mu =\bm{f}.
	\end{equation}
From this equation, it follows, in particular, that:
	\begin{equation}\label{eq: partial^2_3}
		\begin{aligned}
			\mu\partial^2_3\bm{u}'&=-\mu \Delta' \bm{u}' - (\lambda+\mu)\nabla'\textrm{div}\,
			\bm{u}-\nabla'\lambda\,\textrm{div}\,\bm{u}-2(\widehat{\nabla}\bm{u})'\nabla\mu + \bm{f}',\\
			(\lambda+2\mu)\partial^2_3u_3&=-\mu\Delta' u_3 - (\lambda+\mu)\partial_3\, \textrm{div}\,
			\bm{u}'-\partial_3\lambda\, \textrm{div}\bm{u}-2(\widehat{\nabla}\bm{u})_3\cdot \nabla\mu + f_3,
		\end{aligned}
	\end{equation}
where we used the prime notation to denote projection onto the first two variables.
From \eqref{eq: partial^2_3}, since the tangential derivatives are in $H^1_1(\RR^3_-)$, we have that
	\begin{equation}\label{eq: part_3 L2}
		\varrho\mu\partial^2_3\bm{u}'\in L^2(\mathbb{R}^3_-),\qquad \textrm{and}\,\qquad 
		\varrho(\lambda+2\mu)\partial^2_3{u}_3\in L^2(\mathbb{R}^3_-).
	\end{equation}	
From \eqref{eq: part_3 L2}, the bounds on the Lam\'e parameters (see \eqref{ass: c1,1 regularity}), and the strong 
convexity condition \eqref{ass: strong conv}, we obtain   that
	\begin{equation}\label{eq: varrho partial^2_3}
		\varrho\partial^2_3\bm{u}\in L^2(\mathbb{R}^3_-).
	\end{equation}
The regularity estimate \eqref{eq: stab_est_H^2_1} then follows from \eqref{eq: estimate for H^1_1 with f}, the 
assumed regularity of the source and the coefficients, and \eqref{eq: partial^2_3}.
\end{proof}

We now tackle \ref{i.3}, that is, the existence of a very weak solution for Problem  \eqref{eq: reg_prob}. 
More precisely, we will prove the existence of a unique solution $\bm{u}\in E_0(\mathbb{R}^3_-)$  for 
$\bm{f}\in  V'$, where, we recall, the space $V$ is defined in \eqref{eq: space V} and $E_0(\RR^3_-)$ is given in 
\eqref{eq: space E^2_0}.  To do so,  we first establish the validity 
of integration by parts, when $\textrm{div}(\mathbb{C}\widehat{\nabla}\bm{u})\in V'$. 
This result will be a direct consequence of standard integration by parts provided  $\mathcal{D}
(\overline{\mathbb{R}^3_-})$ is dense  into $E_0(\mathbb{R}^3_-)$. The density follows from  an extension to the 
case of weighted Sobolev spaces of results in \cite{Lions-Magenes} and from some  ideas contained in 
\cite{Amrouche-Dambrine-Raudin}. We include the proof of this density result for completeness.

\begin{lemma}\label{lem: density}
The space $\mathcal{D}(\overline{\mathbb{R}^3_-})$ is dense in ${E}_0(\mathbb{R}^3_-)$.
\end{lemma}  

\begin{proof}
By the Riesz Representation Theorem, for every functional $T\in({E}_0(\mathbb{R}^3_-))'$  
there exist $\bm{u}_1\in  H^0_1(\mathbb{R}^3_-)$ and $\bm{u}_2\in V(\mathbb{R}^3_-)$ such that, for every
$\bm{v}\in E_0(\mathbb{R}^3_-)$,
	\begin{equation}\label{eq: formula_density E1}
		\langle T, \bm{v} \rangle=\int_{\mathbb{R}^3_-}\bm{u}_1\cdot \bm{v}\, d\bm{x} + \langle \textrm{div}\,
		\mathbb{C}\widehat{\nabla}\bm{v},\bm{u}_2\rangle_{(V'(\RR^3_-),V(\RR^3_-))}. 
	\end{equation}
Next, we suppose that $T$ is the zero functional when restricted to $\mathcal{D}
(\overline{\mathbb{R}^3_-})$:
	\begin{equation}
		\langle T, \bm{\varphi} \rangle =0,\qquad \forall \bm{\varphi} \in \mathcal{D}(\overline{\mathbb{R}^3_-}).
	\end{equation}
We will show that  $T$ is then the zero functional on $E_0$:
	\begin{equation*}
		\langle T, \bm{v} \rangle =0,\qquad \forall \bm{v} \in E_0({\mathbb{R}}^3_-).
	\end{equation*}
Since $\bm{\varphi}\in \mathcal{D}(\overline{\mathbb{R}^3_-})$, there exists $\bm{\psi}\in \mathcal{D}
(\mathbb{R}^3)$ such that
	\begin{equation*}
		\bm{\psi}_|{_{\overline{\mathbb{R}^3_-}}}=\bm{\varphi}.
	\end{equation*}
Following the approach of Lions-Magenes (see \cite{Lions-Magenes} p. 173), we let $\widetilde{\bm{u}}_1$ and 
$\widetilde{\bm{u}}_2$ denote  the extension  by zero of $\bm{u}_1$ and $\bm{u}_2$ to $\mathbb{R}^3$, 
respectively. We also extend the isotropic elastic tensor $\mathbb{C}$ to an isotropic tensor 
$\widetilde{\mathbb{C}}$ on $\RR^3$,  satisfying
	\begin{equation*}
		\widetilde{\mathbb{C}}\in C^{0,1}(\mathbb{R}^3)\, \qquad \textrm{and}\, \qquad 
		\widetilde{\mathbb{C}}_{|_{\overline{\mathbb{R}^3_-}}}=\mathbb{C}.
	\end{equation*}
Such as extension can be done by even reflection, and, hence, we  can assume 
that  there exist $\widetilde{\alpha}_0>0$ and $\widetilde{\beta}_0>0$ such that
	\begin{equation*}
	\widetilde{\mu}(\bm{x})\geq \widetilde{\alpha}_0>0,\qquad \textrm{and}\qquad 3\widetilde{\lambda}
	(\bm{x})+2\widetilde{\mu}(\bm{x})\geq \widetilde{\beta}_0>0,\qquad \forall\,\bm{x}\in\mathbb{R}^3,
	\end{equation*}
and we can also assume that
	\begin{equation}
		|\nabla \widetilde{\lambda}|\leq \frac{C}{\varrho},\qquad  |\nabla \widetilde{\mu}|\leq \frac{C}
		{\varrho},
	\end{equation} 
almost everywhere in $\mathbb{R}^3$.

Therefore it follows that
	\begin{equation*}
		\begin{aligned}
			\langle T, \bm{\psi}\rangle &=\int_{\mathbb{R}^3} \left[\widetilde{\bm{u}_1}\cdot \bm{\psi} + 
			\textrm{div}(\widetilde{\mathbb{C}}\widehat{\nabla}\bm{\psi}) \cdot \widetilde{\bm{u}}_2 \right]\, 
			d\, \bx \\
			&=\int_{\mathbb{R}^3_-} \left[{\bm{u}_1}\cdot \bm{\varphi}+ \textrm{div}
			(\mathbb{C}\widehat{\nabla}\bm{\varphi})\cdot {\bm{u}}_2\right]\, d\, \bx=0,
		\end{aligned}
	\end{equation*} 
where we used that $\textrm{div}(\widetilde{\mathbb{C}}\widehat{\nabla}\bm{\psi})$ is a bounded function with 
compact support and that $\widetilde{\bm{u}}_2$ is a locally $L^2$ function.
Consequently, for any $\bm{\psi}\in\mathcal{D}(\mathbb{R}^3)$,
	\begin{equation*}
	\langle \widetilde{\bm{u}}_1,\bm{\psi}\rangle + \langle \textrm{div}\,
	\widetilde{\mathbb{C}}\widehat{\nabla}\widetilde{\bm{u}}_2,\bm{\psi}\rangle =0,
	\end{equation*}
from which it follows that
	\begin{equation}\label{eq: u_2 tilde in density proof}
		\textrm{div}(\widetilde{\mathbb{C}}\widehat{\nabla}\widetilde{\bm{u}}_2)=-\widetilde{\bm{u}}_1 
		\quad \textrm{in}\,\,\mathcal{D}'(\mathbb{R}^3).
	\end{equation}	
We next show that $\widetilde{\bm{u}}_2\in H^2_1(\mathbb{R}^3)$ by using the well posedness and  regularity for 
the equation 
\[
                       \textrm{div}(\widetilde{\mathbb{C}}\widehat{\nabla}\widetilde{\bm{u}} )= \tilde{\bm{f}}
\]
on all of $\RR^3$.  The decay condition at infinity  imposed on $\widetilde{\bu}$ as an element of $H^2_1(\mathbb{R}^3)$  (or  just $H^1_0(\mathbb{R}^3)$) ensures the global coercivity of the quadratic form from the strong convexity of the  Lam\'e tensor.  
Therefore, since $\tilde{\bm{f}}=-\widetilde{\bm{u}}_1\in H^0_1(\mathbb{R}^3)$ we can first prove that 
the solution belongs to $H^1_0(\mathbb{R}^3)$ proceeding as in the proof of Theorem \ref{lem: weak_sol}. Then 
following a similar approach as that in the proof of Theorem \ref{th: sol in H^2_1}, we are able to establish that the unique solution, which must agree with $\widetilde{\bm{u}}_2$,   is in $H^2_1(\mathbb{R}^3)$.
Now, since $\widetilde{\bm{u}}_2\in H^2_1(\mathbb{R}^3)$ is an extension by zero of $\bm{u}_2$ in 
$\mathbb{R}^3$, $\bu_2$ and $\nabla \bu_2$ must have trace zero on $\{x_3=0\}$, that is,  
	\begin{equation*}
		\bm{u}_2=\widetilde{\bm{u}}_2{_{|_{\mathbb{R}^3_-}}}\in \mathring{H}^2_{1}(\mathbb{R}^3_-).
	\end{equation*} 	
Exploiting  identity  \eqref{eq: u_2 tilde in density proof}, we can rewrite  \eqref{eq: formula_density E1} as:
\[
    \langle T, \bm{v}\rangle =  - \int_{\mathbb{R}^3_-}\textrm{div}\,(\mathbb{C}\widehat{\nabla}\bm{u}_2)\cdot \bm{v}\, d\bm{x} + \langle \textrm{div}\,(\mathbb{C}\widehat{\nabla}\bm{v}),\bm{u}_2\rangle_{(V'(\RR^3_-),V(\RR^3_-))}.
\]
Since $\mathcal{D}(\mathbb{R}^3_-)$ is dense in $\mathring{H}^2_{1}(\mathbb{R}^3_-)$ by definition, there exists
	\begin{equation}\label{eq: formula_density_E3 }
		{\bm{\varphi}_k}\in\mathcal{D}(\mathbb{R}^3_-)\, \quad \textrm{such that}\,\quad 
		\bm{\varphi}_k\overset{{\tiny{k\to\infty}}}{\longrightarrow}\ \bm{u}_2\,\quad \textrm{in}\,\, 
		\mathring{H}^2_1(\mathbb{R}^3_-).
	\end{equation}	
Hence we find that, for every $\bm{v}\in E_0(\mathbb{R}^3_-)$,
	\begin{equation*}
			\langle T, \bm{v}\rangle 
			=\lim\limits_{k\to \infty}\left[-\int_{\mathbb{R}^3_-}\textrm{div}\,
			(\mathbb{C}\widehat{\nabla}\bm{\varphi}_k)\cdot \bm{v}\, d\bm{x} + \langle \textrm{div}\,
			(\mathbb{C}\widehat{\nabla}\bm{v}),\bm{\varphi}_k\rangle\right] =0.
	\end{equation*}
We conclude that  $\mathcal{D}(\overline{\mathbb{R}^3_-})$ is dense in $E_0(\mathbb{R}^3_-)$ from the 
Hahn-Banach Theorem. In fact, setting for notational convenience $E_1:={\mathcal{D}
(\overline{\mathbb{R}^3_-})}$, we suppose by contradiction that $\overline{E}_1$ is a proper closed subset of  
$E_0(\mathbb{R}^3_-)$. Therefore,  there exists $\bm{f}_0\in E_0(\mathbb{R}^3_-)$ such that $\bm{f}_0 \notin 
\overline{E}_1$, and  we can define a continuous functional $\widehat{T}$ on $\overline{E}_1\cup \{\bm{f}_0\})$ by:
	\begin{equation*}
		\begin{aligned}
		\widehat{T}(\bm{f})&=0\,\qquad \forall \bm{f}\in\overline{E}_1,\\
		\widehat{T}(\bm{f}_0)&=1.
		\end{aligned}
	\end{equation*}	
Then by the Hahn-Banach theorem the operator 
$\widehat{T}$ can be extended to a non-zero functional  $\widehat{T}\in (E_0(\mathbb{R}^3_-))'$, a contradiction, since we proved that any functional that is zero on $E_1$ is zero on $E_0$.
\end{proof}

By the previous lemma and  results in \cite{Amrouche-Dambrine-Raudin}, we can prove the following Green's 
formula.

\begin{proposition}
Let $\bm{u}\in E_0(\mathbb{R}^3_-)$. Then for any $\bm{\phi}\in V(\mathbb{R}^3_-)$, we have
	\begin{equation}\label{eq: green's formulas}
		\begin{aligned}
			\langle \textup{\textrm{div}}(\mathbb{C}\widehat{\nabla}\bm{u}),\bm{\phi} 
			\rangle_{(V'(\mathbb{R}^3_-),V(\mathbb{R}^3_-))}= \int_{\mathbb{R}^3_-}\bm{u}\cdot 
			\textup{\textrm{div}}(\mathbb{C}\widehat{\nabla}\bm{\phi})\, d\bm{x}+ \langle 
			(\mathbb{C}\widehat{\nabla}\bm{u})\bm{e}_3,\bm{\phi}\rangle_{(H^{-3/2}_{-1}(\mathbb{R}^2),\, 
			H^{3/2}_{1}(\mathbb{R}^2))}.
		\end{aligned}
	\end{equation}
\end{proposition}

\begin{proof}
From Lemma \ref{lem: density}, it is sufficient to prove \eqref{eq: green's formulas} for $\bm{u}\in \mathcal{D}(\overline{\mathbb{R}^3_-})$.
Therefore, for any $\bm{u}\in\mathcal{D}(\overline{\mathbb{R}^3_-})$ and $\bm{\phi}\in V(\mathbb{R}^3_-)$ we have 
	\begin{equation}\label{eq: classic green's formula}
			\int_{\mathbb{R}^3_-}\textrm{div}(\mathbb{C}\widehat{\nabla}\bm{u})\cdot \bm{\phi}\, d\bm{x}=\int_{\mathbb{R}^3_-}\bm{u}\cdot \textrm{div}(\mathbb{C}\widehat{\nabla}\bm{\phi})\, d\bm{x}+\int_{\mathbb{R}^2}(\mathbb{C}\widehat{\nabla}\bm{u})\bm{e}_3)\cdot \bm{\phi}\,d\,\bm{x}'.
	\end{equation} 
Next we prove that, if $\bm{u}\in E_0(\mathbb{R}^3_-)$, then 
	\begin{equation}\label{eq: func gamma_N}
		\begin{aligned}
		\gamma_N: E_0(\mathbb{R}^3_-)&\longrightarrow H^{-3/2}_{-1}(\mathbb{R}^2),\\
			\bm{u}&\mapsto (\mathbb{C}\widehat{\nabla}\bm{u})\bm{e}_3,
		\end{aligned}
	\end{equation}
is a linear and continuous functional.
Let $\bm{\zeta}\in H^{3/2}_1(\RR^2)$. By the Trace Theorem,
there exists a lifting function $\bm{\phi}\in H^2_1(\mathbb{R}^3_-)$ such that $\bm{\phi}=\bm{\zeta}$ and 
$(\mathbb{C}\widehat{\nabla}\bm{\phi})\bm{e}_3=0$ on $\mathbb{R}^2$ (see  \cite[Lemma 2.2]{Amrouche-Dambrine-Raudin}); hence $\bm{\phi}\in V(\mathbb{R}^3_-)$. 
From \eqref{eq: classic green's formula}, we obtain
	\begin{equation*}
		\begin{aligned}
			\int_{\mathbb{R}^2}((\mathbb{C}\widehat{\nabla}\bm{u})\bm{e}_3)\cdot\bm{\zeta}\, 
			d\sigma(\bm{x})=\langle \textrm{div}(\mathbb{C}\widehat{\nabla}\bm{u}),
			\bm{\phi}\rangle_{(V'(\mathbb{R}^3_-),V(\mathbb{R}^3_-))}-\langle \bm{u}, \textrm{div}
			(\mathbb{C}\widehat{\nabla}\bm{\phi})\rangle_{(H^0_{-1}
			(\mathbb{R}^3_-),H^0_1(\mathbb{R}^3_-))}.
		\end{aligned}
	\end{equation*}	
The functional  on the left-hand side is, therefore,  well-defined on $H^{3/2}_1(\RR^2)$ as $\bu\in 
\cD(\overline{\RR^3_-})$.
Moreover, the lifting function  $\bm{\phi}$ satisfies
	\begin{equation*}
		\|\bm{\phi}\|_{H^2_1(\mathbb{R}^3_-)}\leq\, c\, \|\bm{\zeta}\|_{H^{3/2}_1(\mathbb{R}^2)},
	\end{equation*} 	
so that
	\begin{equation*}
		\begin{aligned}
			|\langle \gamma_N(\bm{u}),\, \bm{\zeta}\rangle| &\leq C \left[ \|\textrm{div}
			(\mathbb{C}\widehat{\nabla}\bm{u})\|_{V'(\mathbb{R}^3_-)} \|\bm{\phi}\|_{V(\mathbb{R}^3_-)} + \|
			\bm{u}\|_{H^0_{-1}(\mathbb{R}^3_-)} \|\textrm{div}(\mathbb{C}\widehat{\nabla}\bm{\phi})\|
			_{H^0_1(\mathbb{R}^3_-)}   \right]\\
			&\leq C \|\bm{u}\|_{E_0(\mathbb{R}^3_-)} \|\bm{\phi}\|_{H^2_1(\mathbb{R}^3_-)}\leq C \|\bm{u}\|
			_{E_0(\mathbb{R}^3_-)} \|\bm{\zeta}\|_{H^{3/2}_1(\mathbb{R}^2)}.
		\end{aligned}
	\end{equation*}	
That is, \eqref{eq: func gamma_N} holds and the statement of the proposition follows.	
\end{proof}

Next, we prove existence and uniqueness of a very weak solution.

\begin{theorem}[Very weak solution]\label{th: well-posed in H^0_-1}
For any $\bm{f}\in V'(\mathbb{R}^3_-)$, there exists a unique solution $\bm{u}\in H^0_{-1}(\mathbb{R}^3_-)$ to 
Problem \eqref{eq: reg_prob} such that
	\begin{equation*}
		\|\bm{u}\|_{H^0_{-1}(\mathbb{R}^3_-)}\leq\, C\, \|\bm{f}\|_{V'(\mathbb{R}^3_-)}.
	\end{equation*}
\end{theorem}

\begin{proof}
Thanks to the Green's formula \eqref{eq: green's formulas}, for any $\bm{f}\in V'(\mathbb{R}^3_-)$ Problem 
\eqref{eq: reg_prob} is equivalent to the following variational formulation:\\
\textit{find $\bm{u}\in E_0(\mathbb{R}^3_-)$  such that 
	\begin{equation}
		\int_{\mathbb{R}^3_-}\bm{u}\cdot \textup{\textrm{div}}(\mathbb{C}\widehat{\nabla}\bm{v})\, 
		d\bm{x}=\langle \bm{f}, \bm{v}\rangle_{V'(\mathbb{R}^3_-),V(\mathbb{R}^3_-)},\qquad \forall \bm{v}\in 
		V(\mathbb{R}^3_-).
	\end{equation}  
} 
We first note that, from the well-posedness of \eqref{eq: reg_prob} in $H^2_1(\mathbb{R}^3_-)$ (Theorem 
\ref{th: sol in H^2_1}), for any $\bm{\overline{f}}\in H^0_1(\mathbb{R}^3_-)$ there exists $\overline{\bm{v}}\in 
V(\mathbb{R}^3_-)$ satisfying \eqref{eq: reg_prob} with $\bm{f}$ replaced by $\overline{\bm{f}}$ such that
	\begin{equation*}
		\|\overline{\bm{v}}\|_{V(\mathbb{R}^3_-)}= \|\overline{\bm{v}}\|_{H^2_1(\mathbb{R}^3_-)}\leq\, c\, \|
		\bm{\overline{f}}\|_{H^0_1(\mathbb{R}^3_-)}.
	\end{equation*}
The linear functional 
	\begin{equation*}
		\Phi_{\bm{f}}(\bm{\overline{f}})=\langle \bm{f}, \overline{\bm{v}}\rangle_{(V'(\mathbb{R}^3_-),\, 
		V(\mathbb{R}^3_-))}
	\end{equation*}	
is then continuous, as
	\begin{equation*}
		|\Phi_{\bm{f}}(\bm{\overline{f}})|\leq C\, \|\bm{f}\|_{V'(\mathbb{R}^3_-)}\|\overline{\bm{v}}\|
		_{V(\mathbb{R}^3_-)}\leq C \|\bm{f}\|_{V'(\mathbb{R}^3_-)}\|\bm{\overline{f}}\|
		_{H^0_1(\mathbb{R}^3_-)}.
	\end{equation*}	
Consequently, from the  Riesz Representation Theorem there exists a unique $\bm{u}\in H^0_{-1}(\mathbb{R}^3_-)$ 
such that
	\begin{equation*}
		\Phi_{\bm{f}}(\bm{\overline{f}})=\langle \bm{u}, \bm{\overline{f}} \rangle_{(H^0_{-1}(\mathbb{R}^3_-),\, 
		H^0_1(\mathbb{R}^3_-))},\qquad \forall \bm{\overline{f}}\in H^0_1(\mathbb{R}^3_-).
	\end{equation*} 
Since the solution operator of Problem \eqref{le: distr_sol} for strong solutions
	\begin{equation*}
		\begin{aligned}
			\Phi: H^0_1(\mathbb{R}^3_-)&\longrightarrow V(\mathbb{R}^3_-)\\
				\bm{\overline{f}}&\longrightarrow \overline{\bm{v}} 
		\end{aligned}
	\end{equation*}	
is an isomorphism, the assertion of the theorem follows.	
\end{proof}

We lastly address  the well-posedness of  the source  problem, Problem \eqref{le: distr_sol}, by means of 
interpolation. We recall Formula \eqref{eq:interp3.2}, 
\[
      \left[ H^{-2}_{-1} (\RR^3_-), H^{-1}_0(\RR^3_-)\right]_{1-\Theta,2}= H^{-1-\Theta}_{-\Theta}(\RR^3_-). 
\]
The above result is relevant in view of the following auxiliary result.

\begin{proposition}\label{prop: H-2_-1 contained in V'}
If $\bm{f}\in H^{-2}_{-1}(\mathbb{R}^3_-)$ has compact support, then $\bm{f}\in V'(\mathbb{R}^3_-)$. 
\end{proposition}

\begin{proof}
We define $\psi$, a regular cut-off function in $\mathbb{R}^3_-$, such that $\psi=1$ on a compact neighborhood of the support of $\bm{f}$. Then
	\begin{equation*}
	\langle \bm{f}, \bm{v}\rangle_{(V'(\mathbb{R}^3_-), V(\mathbb{R}^3_-))}=\langle \bm{f}\psi, 
	\bm{v}\rangle_{(V'(\mathbb{R}^3_-), V(\mathbb{R}^3_-))}=\langle \bm{f}, \psi\bm{v}\rangle_{(H^{-2}_{-1}
	(\mathbb{R}^3_-), \mathring{H}^2_{1}(\mathbb{R}^3_-))},
	\end{equation*}
is well defined and satisfies
	\begin{equation*}
		\Big|\langle \bm{f}, \bm{v}\rangle_{(V'(\mathbb{R}^3_-), V(\mathbb{R}^3_-))}\Big|\leq \|\bm{f}\|
		_{H^{-2}_{-1}(\mathbb{R}^3_-)}\|\psi\bm{v}\|_{\mathring{H}^2_{1}(\mathbb{R}^3_-)}\leq c\|\bm{f}\|
		_{H^{-2}_{-1}(\mathbb{R}^3_-)}\|\bm{v}\|_{{H}^2_{1}(\mathbb{R}^3_-)}.
	\end{equation*}
The assertion  follows.	
\end{proof}

\begin{remark}\label{rem: well_posedness in H^-2_-1}
By Proposition \ref{prop: H-2_-1 contained in V'}, Theorem \ref{th: well-posed in H^0_-1} also holds for any 
$\bm{f}\in H^{-2}_{-1}(\mathbb{R}^3_-)$ with compact support in $\mathbb{R}^3_-$. 
\end{remark}

We are now in {the} position to establish  the well-posedness of \eqref{le: distr_sol}.

\begin{theorem}\label{th: very_weak_sol}
Problem \eqref{le: distr_sol} with source term given in \eqref{eq: source_term}
has a unique distributional solution $\bm{u}\in H^{{1}/{2}-\varepsilon}_{-{1}/{2}-\varepsilon}(\mathbb{R}^3_-)$, 
for all $\varepsilon>0$.
\end{theorem}

\begin{proof}
We specialize the interpolation Formula \eqref{eq:interp3.2} to the case 
$\Theta=1/2+\varepsilon$ to obtain:
\begin{equation} \label{eq:interp_very_weak_1}
	H^{-{3}/{2}-\varepsilon}_{-{1}/{2}-\varepsilon}(\mathbb{R}^3_-)=[{H^{-2}_{-1}}(\mathbb{R}^3_-), 
	{H^{-1}_{0}}(\mathbb{R}^3_-)]_{\tfrac{1}{2}-\varepsilon,2},
\end{equation}
and similarly Formula \eqref{eq:interp2.1} to the case $\Theta=1/2-\varepsilon$ to obtain:
\begin{equation}\label{eq:interp_very_weak_2}
	H^{{1}/{2}-\varepsilon}_{-{1}/{2}-\varepsilon}(\mathbb{R}^3_-)=[{H^{0}_{-1}}(\mathbb{R}^3_-), 
	{H^{1}_{0}}(\mathbb{R}^3_-)]_{\tfrac{1}{2}-\varepsilon,2}.
\end{equation}
From Proposition \eqref{prop: sobolev_space_source_term}, $\bm{f}_S\in H^{-{3}/{2}-\varepsilon}$ and,
since it has compact support in $\mathbb{R}^3_-$ by  \eqref{eq: source_term}, we deduce that  
	\begin{equation*}
		\bm{f}_S\in H^{-{3}/{2}-\varepsilon}_{-{1}/{2}-\varepsilon}(\mathbb{R}^3_-).
	\end{equation*}
By Theorems \ref{lem: weak_sol} and  \ref{th: well-posed in H^0_-1}, and Remark \ref{rem: well_posedness in H^-2_-1}, we have a bounded 
solution operator $\Phi$, where $\Phi(\bm{f})=\bu$,  mapping
\[ 
     \Phi: V'(\RR^3_-) \to H^0_{-1}(\RR^3_-), \qquad \Phi: (H^{1}_0)'(\RR^3_-) \to H^1_0(\RR^3_-).
\]	
From the hypotheses on the dislocation surface, we can assume that $S\subset \Omega$, where $\Omega$ is a bounded open set such that $\overline{\Omega}\subset \RR^3_-$. We then restricts all source terms $\bm{f}$ to have compact support in 
$\overline{\Omega}$, and we denote by $H^{-s}_{-\alpha,\Omega}(\RR^3_-)$ the closed subspace of $H^{-s}_{-\alpha}(\RR^3_-)$ of distributions with support in $\overline{\Omega}$. Then $\bm{f}_S\in   H^{-{3}/{2}-\varepsilon}_{-{1}/{2}-\varepsilon,\Omega}(\mathbb{R}^3_-)$. Furthermore, since $V'(\mathbb{R}^3_-)\supset H^{-2}_{-1,\Omega}(\mathbb{R}^3_-)$ and $(H^1_0)'(\mathbb{R}^3_-)\supset H^{-1}_{0,\Omega}(\mathbb{R}^3_-)$, from \eqref{eq:interp_very_weak_1} we have:
	\begin{equation*}
		H^{-{3}/{2}-\varepsilon}_{-{1}/{2}-\varepsilon,\Omega}(\mathbb{R}^3_-)= [H^{-2}_{-1,\Omega}(\mathbb{R}^3_-),
		H^{-1}_{0,\Omega}(\mathbb{R}^3_-)]_{\tfrac{1}{2}-\varepsilon,2}\subset [ V'(\RR^3_-),
		{(H^{1}_{0})'}(\mathbb{R}^3_-)]_{\tfrac{1}{2}-\varepsilon,2}.
	\end{equation*} 	
By interpolation, $\Phi$ extends as a continuous solution operator 
\[
     \Phi:  [ V'(\RR^3_-),(H^{1}_{0}(\mathbb{R}^3_-))']_{\tfrac{1}{2}-\varepsilon,2} \to 
       [{H^{0}_{-1}}(\mathbb{R}^3_-) ,
	{H^{1}_{0}}(\mathbb{R}^3_-)]_{\tfrac{1}{2}-\varepsilon,2},
\]
so that, using \eqref{eq:interp_very_weak_2}, we also have:
\[    
    \Phi:  H^{-{3}/{2}-\varepsilon}_{-{1}/{2}-\varepsilon,\Omega}(\mathbb{R}^3_-) \to H^{{1}/{2}-\varepsilon}_{-{1}/{2}-\varepsilon}(\mathbb{R}^3_-),
\]
as a continuous operator, which gives the conclusion of the theorem.
\end{proof}

We close the discussion of well-posedness of the direct problem by showing the equivalence of the transmission 
problem to the source problem.

\begin{lemma}\label{lem: equivalence_two_problems}
 Problem  \eqref{eq: Pu} and Problem \eqref{le: distr_sol} are equivalent. 
\end{lemma}

\begin{proof}
We first observe  that both solutions of \eqref{eq: Pu} and \eqref{le: distr_sol} 
satisfy $\textrm{div}(\mathbb{C}\widehat{\nabla}\bm{u})=\bm{0}$, for all $\bm{x}\in\mathbb{R}^3_-\setminus 
\overline{S}$, and $(\mathbb{C}\widehat{\nabla}\bm{u})\bm{e}_3=\bm{0}$ on $\{x_3=0\}$. Hence we have only to  verify that the jump relations on the dislocation surface $S$ are satisfied.

We take a point $\overline{\bm{x}}\in S$ and a ball $B_{\eta}(\overline{\bm{x}})$, with $\eta$ sufficiently small, 
such that $(B_{\eta}(\overline{\bm{x}})\cap S)\subset S$, since $S$ is an open surface. We indicate with 
$B^+_{\eta}(\overline{\bm{x}})$ ($B^-_{\eta}(\overline{\bm{x}})$) the half ball on the same (opposite) side of the 
unit normal vector $\bm{n}$ on the  boundary of $S$.

To simplify notation,  we define $D^+_{\eta}:=S^+\cap B^+_{\eta}(\overline{\bm{x}})$, $D^-_{\eta}:=S^-\cap 
B^-_{\eta}(\overline{\bm{x}})$ and $D_{\eta}:=S\cap B_{\eta}(\overline{\bm{x}})$ . 

Let $\bm{\varphi}\in \mathcal{D}(B_{\eta}(\overline{\bm{x}}))$ and let $\bm{u}$ be the solution to \eqref{le: 
distr_sol}. We recall that $\bu \in H^{1/2-\varepsilon}$ in a neighborhood of $S$ and that its trace can be defined in 
$H^{-\varepsilon}(S)$ in a weak sense. 
Then, since $\bm{u}$ is a solution to $\textrm{div}(\mathbb{C}\widehat{\nabla}\bm{u})=\bm{0}$ in 
$B^+_{\eta}(\overline{\bm{x}})$ and $B^-_{\eta}(\overline{\bm{x}})$, {by means of Green's formulas as in \cite[Chapter II]{Lions-Magenes}, we have
	\begin{equation}\label{eq: equivalence1}
		\begin{aligned}
			\int\limits_{B^+_{\eta}(\overline{\bm{x}})}\textrm{div}
			(\mathbb{C}\widehat{\nabla}\bm{\varphi})\cdot \bm{u}\, d\bm{x}&\\
			&\hspace{-1.5cm}=-\langle\bm{u}^+, 
			(\mathbb{C}\widehat{\nabla}\bm{\varphi})\bm{n}\rangle_{(H^{-\varepsilon}(D^+_{\eta}),\, 
			H^{\varepsilon}(D^+_{\eta}))}+\langle\, ((\mathbb{C}\widehat{\nabla}\bm{u})\bm{n})^+,\, 
			\bm{\varphi}\rangle_{(H^{-1-\varepsilon}(D^+_{\eta}),\, H^{1+\varepsilon}(D^+_{\eta}))}.
		\end{aligned}
	\end{equation}
Analogously,	  
	\begin{equation}\label{eq: equivalence2}
		\begin{aligned}
			\int\limits_{B^-_{\eta}(\overline{\bm{x}})}\textrm{div}(\mathbb{C}\widehat{\nabla}\bm{\varphi})\cdot 
			\bm{u}\, d\bm{x}&\\
			&\hspace{-1.5cm}=\langle\bm{u}^-, 
			(\mathbb{C}\widehat{\nabla}\bm{\varphi})\bm{n}\rangle_{(H^{-\varepsilon}(D^-_{\eta}),\, 
			H^{\varepsilon}(D^-_{\eta}))}-\langle\, ((\mathbb{C}\widehat{\nabla}\bm{u})\bm{n})^-,\, 
			\bm{\varphi}\rangle_{(H^{-1-\varepsilon}(D^-_{\eta}),\, H^{1+\varepsilon}(D^-_{\eta}))}.
		\end{aligned}
	\end{equation}
}	
Summing \eqref{eq: equivalence1} and \eqref{eq: equivalence2} gives:
	\begin{equation}\label{eq: equivalence3}
		\begin{aligned}
			\int\limits_{B_{\eta}(\overline{\bm{x}})}\textrm{div}(\mathbb{C}\widehat{\nabla}\bm{\varphi})\cdot 
			\bm{u}\, d\bm{x}&\\
			&\hspace{-1.5cm}=- \langle\, [\bm{u}]_S,\, 
			(\mathbb{C}\widehat{\nabla}\bm{\varphi})\bm{n}\rangle_{(H^{-\varepsilon}(D_{\eta}),\, 
			H^{\varepsilon}(D_{\eta}))}+\langle\, [(\mathbb{C}\widehat{\nabla}\bm{u})\bm{n}],\, \bm{\varphi}\,
			\rangle_{(H^{-1-\varepsilon}(D_{\eta}),\, H^{1+\varepsilon}(D_{\eta}))},
		\end{aligned}
	\end{equation}
where $[\cdot]_S$ denotes the jump on $S$.
Since $\bu$ is a solution of  \eqref{le: distr_sol} with source term \eqref{eq: source_term}, it follows that
	\begin{equation}\label{eq: equivalence4}
		\begin{aligned}
		 	\int\limits_{B^+_{\eta}(\overline{\bm{x}})}\textrm{div}(\mathbb{C}\widehat{\nabla}\bm{\varphi})\cdot 
		 	\bm{u}\, d\bm{x}&+	\int\limits_{B^-_{\eta}(\overline{\bm{x}})}\textrm{div}
		 	(\mathbb{C}\widehat{\nabla}\bm{\varphi})\cdot \bm{u}\, d\bm{x}		 	=\int\limits_{B_{\eta}
		 	(\overline{\bm{x}})}\textrm{div}(\mathbb{C}\widehat{\nabla}\bm{\varphi})\cdot \bm{u}\, d\bm{x}\\
		 	&=\langle \textrm{div}(\mathbb{C}\widehat{\nabla}\bm{u}),\, \bm{\varphi} 
		 	\rangle_{(H^{-{3}/{2}-\varepsilon}(B_{\eta}(\overline{\bm{x}})),\, H^{{3}/{2}+\varepsilon}(B_{\eta}
		 	(\overline{\bm{x}})))}\\
		 	&=-\langle \mathbb{C}(\bm{g}\otimes \bm{n})\delta_S,\, 
		 	\nabla\bm{\varphi}\rangle_{(H^{-{1}/{2}-\varepsilon}(B_{\eta}(\overline{\bm{x}})),\, 
		 	H^{{1}/{2}+\varepsilon}(B_{\eta}(\overline{\bm{x}})))}\\
		 	&=-\int\limits_{D_{\eta}}\nabla\bm{\varphi}:\mathbb{C}(\bm{g}\otimes\bm{n})\, d\sigma(\bm{x}).
		\end{aligned}
	\end{equation}
By the symmetries the tensor $\mathbb{C}$ satisfies, we also have that 
	\begin{equation*}
		\nabla\bm{\varphi}:\mathbb{C}
		(\bm{g}\otimes\bm{n})=\bm{g}\cdot(\mathbb{C}\widehat{\nabla}\bm{\varphi})\bm{n}.
	\end{equation*} 
Hence, equation \eqref{eq: equivalence4} becomes
	\begin{equation}\label{eq: equivalence5}
	\int\limits_{B_{\eta}(\overline{\bm{x}})}\textrm{div}(\mathbb{C}\widehat{\nabla}\bm{\varphi})\cdot \bm{u}\, 
	d\bm{x}=-\int\limits_{D_{\eta}}\bm{g}\cdot(\mathbb{C}\widehat{\nabla}\bm{\varphi})\bm{n}\, 
	d\sigma(\bm{x}).
	\end{equation}	
Comparing \eqref{eq: equivalence3} and \eqref{eq: equivalence5} we have that
	\begin{equation}
	\begin{aligned}
	-\int\limits_{D_{\eta}}\bm{g}\cdot(\mathbb{C}\widehat{\nabla}\bm{\varphi})\bm{n}\, d\sigma(\bm{x})&\\
	&\hspace{-1.5cm}=- 
	\langle\, [\bm{u}],\, (\mathbb{C}\widehat{\nabla}\bm{\varphi})\bm{n}\rangle_{(H^{-\varepsilon}(D_{\eta}),\, 
	H^{\varepsilon}(D_{\eta}))}+\langle\, [(\mathbb{C}\widehat{\nabla}\bm{u})\bm{n}],\, \bm{\varphi}\,
	\rangle_{(H^{-1-\varepsilon}(D_{\eta}),\, H^{1+\varepsilon}(D_{\eta}))},
	\end{aligned}
	\end{equation}
for any $\bm{\varphi} \in \cD(B_\eta(\Bar{\bm{x}}))$. By taking $\bm{\varphi}$ constant near $\Bar{\bm{x}}$, we conclude from the identity 
above that  $[(\mathbb{C}\widehat{\nabla}\bm{u})\bm{n}]=\bm{0}$ in $D_{\eta}$, and hence it follows 
$[(\mathbb{C}\widehat{\nabla}\bm{u})\bm{n}]=\bm{0}$ in $S$.  Then, $[\bm{u}]=\bm{g}$  in $S$.
We have shown that, if $\bu$ is a solution of \eqref{le: distr_sol} , it is also a very weak solution of \eqref{eq: Pu}. The converse implication follows by simply reversing all arguments in the proof.
\end{proof}

From Theorem \ref{th: very_weak_sol} and the previous lemma, we finally have the well-posedness of the direct problem in $\RR^3_-$.

\begin{corollary}
There exists a unique very weak solution $\bm{u}\in H^{{1}/{2}-\varepsilon}_{-{1}/{2}-\varepsilon}(\mathbb{R}^3_-)$ of the boundary-value/transmission  problem \eqref{eq: Pu}.
\end{corollary}	

\section{The solution as a double layer potential}
In this section, we prove the existence of a Neumann function in the half-space, for an isotropic, non-homogeneous, elastic tensor satisfying \eqref{ass: c1,1 regularity} and \eqref{eq: strong convexity}. The Neumann function is utilized to give a representation of the solution $\bm{u}$ to Problem \eqref{eq: Pu} as a double layer potential.  

\subsection{The Neumann function}
In this subsection, we prove the existence 
of a distributional solution to the problem:
\begin{equation}\label{eq: p_neumann}
	\begin{cases}
		\textrm{\textup{div}}\,(\mathbb{C}(\bm{x})\widehat{\nabla}\bf N(\bm{x},\bm{y}))=\delta_{\bm{y}}
		(\bm{x})\bf I, & \textrm{in}\, \, \mathbb{R}^3_-,\\
		(\mathbb{C}(\bm{x})\widehat{\nabla}\bf N(\bm{x},\bm{y}))\bm{e}_3=\bm{0} ,& \textrm{on}\, \, \{x_3=0\},
	\end{cases}
\end{equation} 
where $\delta_{\bm{y}}(\cdot)$ is the Dirac distribution supported at $\bm{y}\in\mathbb{R}^3_-$.
To prove the existence of the Neumann function, we work column-wise and consider the system
\begin{equation}\label{eq: Neum_equ}
	\textrm{div}\,(\mathbb{C}(\bm{x})\widehat{\nabla}{\bm{N}}^{(k)}(\bm{x},\bm{y}))=\bm{e}_k 
	\delta_{\bm{y}}(\bm{x}),\qquad \textrm{in}\, \mathbb{R}^3_-,
\end{equation}
for $k=1,2,3$.	
We observe that, by freezing the coefficients of the elastic tensor at $\bm{x}=\bm{y}$, we formally have 
$\bm{e}_k\delta_{\bm{y}}(\bm{x})=\textrm{div}\,(\mathbb{C}
(\bm{y})\widehat{\nabla}\widetilde{\bm{N}}^{(k)}(\bm{x},\bm{y}))$, where $\widetilde{\bm{N}}^{(k)}$ is 
the $k$-th column vector of the Neumann function for the Lam\'e system with constant coefficients.  Therefore,  
subtracting $\textrm{div}(\mathbb{C}(\bm{x})\widehat{\nabla}\widetilde{\bm{N}}^{(k)})$ from both sides in 
Equation \eqref{eq: Neum_equ} gives:
\begin{equation}\label{eq: eq_M}
	\textrm{div}\,(\mathbb{C}(\bm{x})\widehat{\nabla}\bm{M}^{(k)}(\bm{x},
	\bm{y}))=-\textrm{div}\left[(\mathbb{C}(\bm{x})-\mathbb{C}
	(\bm{y}))\widehat{\nabla}\widetilde{\bm{N}}^{(k)}(\bm{x},\bm{y})\right],\qquad \textrm{in}\,
	\mathbb{R}^3_-,
\end{equation} 
where $\bm{M}^{(k)}:=\bm{N}^{(k)}-\widetilde{\bm{N}}^{(k)}$, for $k=1,2,3$. We next recall the 
decay estimates satisfied by the Neumann function $\mathbf{\widetilde{N}}$ in the case of constant coefficients. In 
particular,  it is not difficult to see from Theorem 4.9 in \cite{Aspri-Beretta-Rosset} that there exists a positive 
constant $C=C( \alpha_0,\beta_0,M)$, such that for all $\bm{x},\bm{y}\in \mathbb{R}^3_-$ with $\bm{x}\neq 
\bm{y}$,
\begin{equation}\label{eq: decay_cond_Neumann}
	\begin{aligned}
		|\mathbf{\widetilde{N}}(\bm{x},\bm{y})|&\leq C |\bm{x}-\bm{y}|^{-1},\\
		|\nabla_{\bm{x}}\mathbf{\widetilde{N}}(\bm{x},\bm{y})|&\leq C|\bm{x}-\bm{y}|^{-2}.
	\end{aligned}
\end{equation}
We recall that $\alpha_0$, $\beta_0$, $M$ are the constants appearing in the assumptions on the elasticity tensor 
$\CC$ in Subsection \ref{sec:main_assumptions}.

Next, we establish rigorously the existence of the Neumann function $\mathbf{N}$, by showing that there exists a unique  variational solution $\mathbf{M}$ for the vector problem \eqref{eq: eq_M} in $H^1_0(\RR^3_-)$.
This result also implies that, as expected, the singularities of $\mathbf{N}(\bx,\by)$ near $\by$  are those of the constant-coefficient Neumann function obtained by freezing the coefficient at $\by$. This fact will be used in Subsection \ref{sec:representation}. 

\begin{proposition}\label{lem: F in L2}
 Assume that  \eqref{ass: c1,1 regularity} holds, and let 
	\begin{equation}\label{eq: beh_Neum_func}
		\mathbf{F}_{\bm{y},k}(\bm{x}):=[(\mathbb{C}(\bm{x})-\mathbb{C}
		(\bm{y}))\widehat{\nabla}\widetilde{\bm{N}}^{(k)}(\bm{x},\bm{y})], \quad k=1,2,3.
	\end{equation}
Then $\mathbf{F}_{\bm{y},k}\in L^2(\mathbb{R}^3_-)$ for any $\bm{y}\in\mathbb{R}^3_-$.
Moreover, for any $\bm{y}\in\mathbb{R}^3_-$ the boundary value problem 
\begin{equation}\label{eq: prob_M}
		\begin{cases}
			\textup{\textrm{div}}\,(\mathbb{C}(\bm{x})\widehat{\nabla}\bm{M}^{(k)}(\bm{x},
			\bm{y}))=-\textup{\textrm{div}}\,\mathbf{F}_{\bm{y},k}, & \textup{\textrm{in}}\,
			\mathbb{R}^3_,\\
			\mathbb{C}(\bm{x})\widehat{\nabla}\bm{M}^{(k)}\bm{e_3}=-\mathbb{C}
			(\bm{x})\widehat{\nabla}\widetilde{\bm{N}}^{(k)}\bm{e}_3,& \textup{\textrm{on}}\, \{x_3=0\}, 	
		\end{cases}
\end{equation}
admits a unique solution, satisfying
\begin{equation}\label{eq: stab_est_M^k}
		\|\bm{M}^{(k)}\|_{H^{1}_0(\mathbb{R}^3_-)}\leq C \|\mathbf{F}_{\bm{y},k}\|
		_{L^2(\mathbb{R}^3_-)}.
\end{equation} 
In particular the matrix $\mathbf{M}=[\bm{M}^{(1)}\, 
\bm{M}^{(2)}\, \bm{M}^{(3)}]$ belongs to $H^{1}_0(\mathbb{R}^3)$.
\end{proposition}

\begin{proof}
We start by showing that $\mathbf{F}_{\bm{y},k}\in L^2(\mathbb{R}^3_-)$.  We
choose $r$ sufficiently small so that $B_{r}(\bm{y})\subset \mathbb{R}^3_-$.
From the regularity assumption on the elasticity tensor \eqref{ass: c1,1 regularity}, it follows that there exists  a 
positive constant $C$ independent of $\by$ and $r$ such that, for $\bm{x}\in B_{r}(\bm{y})$,
	\begin{equation*}
		|\mathbb{C}(\bm{x})-\mathbb{C}(\bm{y})|\leq C |\bm{x}-\bm{y}|.
	\end{equation*}
By \eqref{eq: decay_cond_Neumann} then,
	\begin{equation*}
		\int\limits_{B_{r}(\bm{y})}|[\mathbb{C}(\bm{x})-\mathbb{C}
		(\bm{y})]\widehat{\nabla}\widetilde{\bm{N}}^{(k)}|^2\, d\bm{x}\leq C \int\limits_{B_{r}
		(\bm{y})}\frac{1}{|\bm{x}-\bm{y}|^2}\, d\bm{x}<\infty.
	\end{equation*}	
On the other hand,  \eqref{ass: c1,1 regularity} also implies the elastic parameters are uniformly bounded and, 
again by \eqref{eq: decay_cond_Neumann}, we find that 
	\begin{equation*}
		\int\limits_{B^C_{r}(\bm{y})\cap\mathbb{R}^3_-}|[\mathbb{C}(\bm{x})-\mathbb{C}
		(\bm{y})]\widehat{\nabla}\widetilde{\bm{N}}^{(k)}|^2\, d\bm{x}\leq C \int\limits_{B^C_{r}
		(\bm{y})\cap\mathbb{R}^3_-}\frac{1}{|\bm{x}-\bm{y}|^4}\, d\bm{x}<\infty,
	\end{equation*}	 
where $B^C_{r}(\bm{y})$ is the complementary set of $B_{r}(\bm{y})$. Combining these two results gives that 
$\mathbf{F}_{\bm{y},k}(\bm{x})\in L^2(\mathbb{R}^3_-)$ for any $\bm{y}\in\mathbb{R}^3_-$.

Recalling that
	\begin{equation*}
		(\mathbb{C}(\bm{y})\widehat{\nabla}\bf \widetilde{N}(\bm{x},\bm{y}))\bm{e}_3=\bm{0} ,\,\, 
		\textrm{on}\, \, \{x_3=0\},\,
	\end{equation*}
we can rewrite (\ref{eq: prob_M}) in the equivalent form 
\begin{equation}\label{eq: equivprob_M_new}
	\begin{cases}
		\textup{\textrm{div}}\,(\mathbb{C}(\bm{x})\widehat{\nabla}\bm{M}^{(k)}(\bm{x},
		\bm{y})+\mathbf{F}_{\bm{y},k})=\bm{0}, & \textup{\textrm{in}}\,\mathbb{R}^3_-,\\
		(\mathbb{C}(\bm{x})\widehat{\nabla}\bm{M}^{(k)}+\mathbf{F}_{\bm{y},k})\, \bm{e_3}=\bm{0},& 
		\textup{\textrm{on}}\, \{x_3=0\}.
		\end{cases}
	\end{equation}
Proceeding as in the first step of the  proof of Theorem \ref{th: sol in H^2_1} (cf. Problem \ref{eq: bvp_rho_partial_u_new}), the existence and regularity of $\bm{M}^{(k)}$ as the unique solution of \eqref{eq: prob_M} follows from the well-posedness of the variational formulation of  \eqref{eq: equivprob_M_new}, i.e., find $\bm{w} \in H^1_0(\RR^3_-)$ such that
\[
      a(\bm{w},\bm{v}) = G_k(\bm{v}), \qquad \forall \bm{v}\in H^1_0(\RR^3_-), 
\]
where the functional
	\begin{equation}\label{eq: funct}
		G_k(\bm{v})=-\int\limits_{\mathbb{R}^3_-}(\mathbb{C}(\bm{x})-\mathbb{C}
		(\bm{y}))\widehat{\nabla}\widetilde{\bm{N}}^{(k)}(\bm{x},\bm{y}):\widehat{\nabla}\bm{v}(\bm{x})\, 
		d\bm{x}.
	\end{equation}	
By Lax-Milgram, it is enough to show that $G_k$ is continuous on $H^1_0(\RR^3_-)$. Since $\mathbf{F}_{\by,k}$ was shown to belong to $L^2(\RR^3_-)$, we  have that
	\begin{equation*}
		\begin{aligned}
			|G_k(\bm{v})|&\leq \int\limits_{\mathbb{R}^3_-}|\mathbf{F}_{\bm{y},k}:\widehat{\nabla}\bm{v}|\, 
			d\bm{x}\leq \|\mathbf{F}_{\bm{y},k}\|_{L^2(\mathbb{R}^3_-)}\|\widehat{\nabla}\bm{v}\|
			_{L^2(\mathbb{R}^3_-)}\leq \|\mathbf{F}_{\bm{y},k}\|_{L^2(\mathbb{R}^3_-)} \|\bm{v}\|
			_{H^{1}_0(\mathbb{R}^3_-)}. \\
		\end{aligned}
	\end{equation*}	
\end{proof}

\begin{remark}
If $\textup{\textrm{dist}}(\bm{y},\{x_3=0\})\geq d_0>0$,  then from \eqref{eq: stab_est_M^k} it follows that
	\begin{equation*}
		\|\bm{M}^{(k)}\|_{H^{1}_0(\mathbb{R}^3_-)}\leq C
	\end{equation*}
where $C$ does not depend on $\bm{y}\in\mathbb{R}^3_-$.	 
\end{remark}

In the next section, to provide an integral representation formula of the solution $\bm{u}$ of Problem \eqref{eq: Pu} as a double layer potential, we need to prove higher regularity than $H^1_0$ on the Neumann function once we are sufficiently far from the singularity $\bm{y}$. 

\begin{proposition}\label{prop: N_dag}
For any $r>0$ such that $B_{r}(\bm{y})\subset\mathbb{R}^3_-$,  we have that
	\begin{equation*}
		\mathbf{N}\in H^2_1(\mathbb{R}^3_-\setminus \overline{B_{r}(\bm{y})}).
	\end{equation*}
\end{proposition} 

\begin{proof}
We fix $r>0$ such that $B_{r}(\bm{y})\subset \mathbb{R}^3_-$, and define a cut-off function $\varphi\in \mathcal{D}(\mathbb{R}^3_-)$ with the property that
	\begin{equation*}
	\varphi=
		\begin{cases}
			1, & \textrm{in}\, \overline{B_{{r}/{2}}(\bm{y})},\\
			0, & \textrm{in}\, \mathbb{R}^3_-\setminus B_{r}(\bm{y}).
		\end{cases}
	\end{equation*}
We also let
	\begin{equation}\label{eq: N_dag}
		\bm{N}^{(k)}_{\dag}:=(1-\varphi)\bm{N}^{(k)}=
			\begin{cases}
				\bm{0}, & \textrm{in}\, \overline{B_{{r}/{2}}(\bm{y})}, \\
				\bm{N}^{(k)}, & \textrm{in}\, \mathbb{R}^3_-\setminus B_{r}(\bm{y}),
			\end{cases}
	\end{equation}
for $k=1,2,3$.	
From the definition of $\varphi$ and the fact that $\bm{N}^{(k)}$ solves the homogeneous equation $\textrm{div}(\mathbb{C}\widehat{\nabla}\bm{N}^{(k)})=\bm{0}$ for $\bm{x}\neq \bm{y}$, it is straightforward to find the equation solved by $\bm{N}^{(k)}_{\dag}$, that is
	\begin{equation*}
		\textrm{div}(\mathbb{C}\widehat{\nabla}\bm{N}^{(k)}_{\dag})=-\textrm{div}(\mathbb{C}(\widehat{\nabla\varphi \otimes \bm{N}^{(k)}}))-(\mathbb{C}\widehat{\nabla}\bm{N}^{(k)})\nabla{\varphi},	
	\end{equation*}
with homogeneous Neumann boundary conditions.
We observe that the source term
	\begin{equation*}
		\bm{h}:= -\textrm{div}(\mathbb{C}(\widehat{\nabla\varphi \otimes \bm{N}^{(k)}}))-(\mathbb{C}\widehat{\nabla}\bm{N}^{(k)})\nabla{\varphi}
	\end{equation*}	
has compact support in $\overline{B_{r}(\bm{y})}\setminus B_{r/2}(\bm{y})$ and, moreover, $\bm{N}^{(k)}\in H^1(B_{r}(\bm{y})\setminus \overline{B_{r/2}(\bm{y})})$, which follows from the result in Proposition \ref{lem: F in L2} and the representation  $\bm{N}^{(k)}=\bm{M}^{(k)}-\widetilde{\bm{N}}^{(k)}$. Therefore, $\bm{h}\in L^2(B_{r}(\bm{y})\setminus \overline{B_{r/2}(\bm{y})})$ and, since $\bm{h}$ has compact support, $\bm{h}\in H^0_1(\mathbb{R}^3_-)$ as well.
We then consider the problem: 
	\begin{equation*}
		\begin{cases}
			\textrm{div}(\mathbb{C}\widehat{\nabla}\bm{N}^{(k)}_{\dag})=\bm{h}, & \textrm{in}\, \mathbb{R}^3_-,\\
			(\mathbb{C}\widehat{\nabla}\bm{N}^{(k)}_{\dag})\bm{e}_3=\bm{0}, & \textrm{on}\, \{x_3=0\}
		\end{cases}
	\end{equation*} 
for given  $\bm{h}\in H^0_1(\mathbb{R}^3_-)$.
Following the steps in the proof of Theorem \ref{lem: weak_sol} and Theorem \ref{th: sol in H^2_1}, one can prove that there exists a unique $\bm{N}^{(k)}_{\dag}\in H^2_1(\mathbb{R}^3_-)$, that is,  $\bm{N}^{(k)}\in H^2_1(\mathbb{R}^3_-\setminus \overline{B_{r}(\bm{y})})$ from \eqref{eq: N_dag},  for $k=1,2,3$.  		
\end{proof}

\subsection{A representation formula for the solution to  \eqref{eq: Pu} } \label{sec:representation}

In this subsection, using the Neumann function defined in \eqref{eq: p_neumann}, we give an integral representation  formula for the solution to Problem \eqref{eq: Pu}. 
Then we take advantage of this integral representation to study the regularity of the solution in the complement of 
the dislocation surface $S$ in $\RR^3_-$. In fact, we will determine  the singularities of the solution, when  $S$ is a rectangular 
dislocation surface parallel to the plane $\{x_3=0\}$ and $\bm{g}$ is a constant vector, in the special case that the medium is 
homogeneous.

This explicit example shows that, if $\bm{g}\in H^{1/2}(S)$, but without assuming that $\bg$ has compact  support 
in $S$, generically  solutions to \eqref{eq: Pu}  are not in  $H^1_0(\mathbb{R}^3_-\setminus \overline{S})$. 
We begin with a preliminary result proved by following an approach similar to that in \cite{ColliFranzone-Guerri-Magenes} and using also results in \cite{Fuchs84} on growth properties of Neumann functions in a neighborhood of the singularity.

\begin{proposition}
The unique solution to \eqref{eq: Pu} can be represented as a double layer potential on $S$, that is,
		\begin{equation}\label{eq: repr_u_var_coeff}
			\bm{u}(\bm{y})=-\int\limits_{S}\left[(\mathbb{C}(\bm{x})\widehat{\nabla}_{\bm{x}}\bf N (\bm{x},
			\bm{y})) \bm{n}(\bm{x})\right]^T \bm{g}(\bm{x})\, d\sigma(\bm{x}),
		\end{equation}
where $\mathbf{N}$ is the Neumann function satisfying \eqref{eq: p_neumann}.
\end{proposition}	

\begin{proof}
We recall that the transmission problem \eqref{eq: Pu} is equivalent to the source problem \eqref{le: distr_sol}, so we provide the integral representation formula starting from \eqref{le: distr_sol}.
From regularity results for elliptic systems (see e.g. \cite{Li-Nirenberg}), it is immediate that the solution $\bu$ is regular in $\RR^3_-\setminus \overline{S}$ and has traction zero on $\{x_3=0\}$. 

We fix $\bm{y}\in \mathbb{R}^3_-\setminus \overline{S}$ and we consider a ball $B_{r}(\bm{y})$ such that $\overline{B_{r}(\bm{y})}\subset \mathbb{R}^3_-$ with $B_{r}(\bm{y})\cap \overline{S}=\emptyset$. 

From Proposition \ref{prop: sobolev_space_source_term}, $\bm{f}_S\in H^{-3/2-\varepsilon}$ and  has compact support in $\mathbb{R}^3_-$. Without loss of generality, we assume the support of $\bm{f}_S$ lies in an open set $\Omega\subset \RR^3_-$, the closure of which does not meet $\RR^2$ nor the boundary of the ball $B_r(\bm{y})$.   We then have from Proposition \ref{prop: H-2_-1 contained in V'} that
	\begin{equation*}
		\bm{f}_S\in H^{-3/2-\varepsilon}_{-1/2-\varepsilon,\Omega}\subset H^{-2}_{-1,\Omega}\subset (H^2_1)'\subset V',
	\end{equation*} 
where $H^{-s}_{-\alpha,\Omega}$ are the spaces of distributions with compact support in $\overline{\Omega}$.
Then, $\bu$ also solves the source problem in $V'$, that is,  $\bu\in E_0(\RR^3)$. 
Moreover, for $k=1,2,3$,  $\bm{N}^{(k)}\in H^2_1(\mathbb{R}^3_-\setminus \overline{B_{r}(\bm{y})})$ from Proposition \ref{prop: N_dag}, and  $(\mathbb{C}\widehat{\nabla}\bm{N}^{(k)})\bm{e}_3=\bm{0}$ on $\{x_3=0\}$ by hypothesis. 
We observe that we can then apply Green's formula \eqref{eq: green's formulas} in $\mathbb{R}^3_-\setminus \overline{B_{r}(\bm{y})}$
with $\bm{N}^{(k)}$ as test function. (That formula is derived in $\RR^3_-$, but it can be extended to $\RR^3_-\setminus  \overline{B_{r}(\bm{y})})$ in this case, since both $\bu$ and $\bf{N}$ are regular near $\partial B_{r}(\bm{y})$.)

As a result, we obtain that 
	\begin{equation}\label{eq: rapr_rough}
		\begin{aligned}
		&\int_{\partial B_{r}(\bm{y})}(\mathbb{C}(\bm{x})\widehat{\nabla}\bm{N}^{(k)}(\bm{x},\bm{y}))\bm{n}\cdot\bm{u}(\bm{x})\, d\sigma(\bm{x})-\int_{\partial B_{r}(\bm{y})}(\mathbb{C}(\bm{x})\widehat{\nabla}\bm{u}(\bm{x}))\bm{n}\cdot\bm{N}^{(k)}(\bm{x},\bm{y})\, d\sigma(\bm{x})\\
		&=\langle \bm{f}_S,\bm{N}^{(k)}(\cdot,\bm{y})\rangle_{(H^2_1)'(\mathbb{R}^3_-\setminus \overline{B_{r}(\bm{y})}), H^2_1(\mathbb{R}^3_-\setminus \overline{B_{r}(\bm{y})})}\\
		&=-\langle \mathbb{C}(\bm{g}\otimes \bm{n})\delta_S,\nabla\bm{N}^{(k)}(\cdot,\bm{y})\rangle_{(H^{1}_{1})'(\mathbb{R}^3_-\setminus \overline{B_{r}(\bm{y})}), H^{1}_{1}(\mathbb{R}^3_-\setminus \overline{B_{r}(\bm{y})})},	
		\end{aligned}
	\end{equation}
where we used that  $\textrm{div}(\mathbb{C}\widehat{\nabla}\bm{N}^{(k)}(\bm{x},\bm{y}))=\bm{0}$ in $\mathbb{R}^3_-\setminus \overline{{B_{r}(\bm{y})}}$ and the traction-free boundary condition.
The last equality above follows from the fact that 
	\begin{equation*}
		 \mathbb{C}(\bm{g}\otimes \bm{n})\delta_S\in H^{-1/2-\varepsilon}_{-1/2-\varepsilon,\Omega}(\mathbb{R}^3_-\setminus \overline{B_{r}(\bm{y})})\subset (H^{1/2+\varepsilon}_{1/2+\varepsilon})'(\mathbb{R}^3_-\setminus \overline{B_{r}(\bm{y})})\subset (H^1_1)'(\mathbb{R}^3_-\setminus \overline{B_{r}(\bm{y})}). 
	\end{equation*}
As discussed in \cite{Fuchs84,Fuchs86}, the Neumann function  admits the decomposition:
\[
    {\bf{N}}(\bm{x},\bm{y})={\bf{\Gamma}}(\bm{x}-\bm{y})+{\bf{H}}_{\bm{y}}^{(k)}(\bm{x},\bm{y}),
\]
where ${\bf{\Gamma}}$ denotes the fundamental solution for the Lam\'e operator with constant coefficients (we freeze the coefficients at $\bm{y}$) and ${\bm{H}}_{\bm{y}}$ is a more regular remainder. Using the Lipschitz continuity of the Lam\'e coefficients, it is possible to prove that there exist $C_1$, $C_2$, $C_3$, $C_4>0$, which depend on the constants in the a priori assumptions \eqref{ass: c1,1 regularity} and \eqref{eq: strong convexity} on $\CC$, and on the distance of the fixed point $\bm{y}$ from $\{x_3=0\}$, such that 
	\begin{equation}\label{eq: estim_Neumann}
		\begin{aligned}
		|\bm{N}^{(k)}(\bm{x},\bm{y})|&\leq C_1 |\bm{x}-\bm{y}|^{-1}+ C_2\\
		|\nabla_{\bm{x}}\bm{N}^{(k)}(\bm{x},\bm{y})|&\leq C_3 |\bm{x}-\bm{y}|^{-2} + C_4 |\bm{x}-\bm{y}|^{-1}.
		\end{aligned}
	\end{equation} 
Following the same calculations as, for example, in \cite[Theorem 3.3]{Aspri-Beretta-Mascia} and employing the local estimates  \eqref{eq: estim_Neumann}, it is straightforward to prove that 
	\begin{equation}\label{eq: est_rapr}
		\Bigg|\int_{\partial B_{r}(\bm{y})}(\mathbb{C}\widehat{\nabla}\bm{u})\bm{n}\cdot \bm{N}^{(k)}\, d\sigma(\bm{x})\Bigg| \to 0,\, \, \textrm{as}\,\, r\to 0,
	\end{equation} 
and 
	\begin{equation}\label{eq: uk_rapr}
		\int_{\partial B_{r}(\bm{y})}(\mathbb{C}\widehat{\nabla}\bm{N}^{(k)})\bm{n}\cdot \bm{u}\, d\sigma(\bm{x}) \to u_k(\bm{y}),\, \, \textrm{as}\,\, r\to 0,
	\end{equation} 
where $u_k$ is $k$-th component of the displacement vector $\bm{u}$.	
To handle the last term in \eqref{eq: rapr_rough}, we use the density of the space $\mathcal{D}(\overline{\mathbb{R}^3_-}\setminus B_{r}(\bm{y}))$ in $H^1_1(\mathbb{R}^3_-\setminus \overline{B_{r}(\bm{y})})$. By density, there exist
	\begin{equation*}
		\{\bm{\eta}^{(k)}_j\}\in \mathcal{D}(\overline{\mathbb{R}^3_-}\setminus B_{r}(\bm{y}))\,\, \textrm{such that}\, \, \bm{\eta}^{(k)}_j \overset{j\to\infty}{\longrightarrow} \bm{N}^{(k)}(\cdot,\bm{y})\,\, \textrm{in}\,\,  H^2_1(\mathbb{R}^3_-\setminus \overline{B_{r}(\bm{y})}), 
	\end{equation*}
hence $\nabla\bm{\eta}^{(k)}_j \to \nabla\bm{N}^{(k)}(\cdot,\bm{y})$ in $H^{1/2}(S)$, for $k=1,2,3$. Therefore,
	\begin{equation}\label{eq: limit_rapr}
		\lim\limits_{j\to \infty}\int_{S}\mathbb{C}(\bm{x})(\bm{g}(\bm{x})\otimes \bm{n}(\bm{x})): \nabla \bm{\eta}^{(k)}_j(\bm{x})\, d\sigma(\bm{x})=\int_{S} \mathbb{C}(\bm{x})(\bm{g}(\bm{x})\otimes \bm{n}(\bm{x})): \nabla\bm{N}^{(k)}(\bm{x},\bm{y})\, d\sigma(\bm{x}),
	\end{equation}	
where we used that, by \eqref{eq: def_hdelta},
	\begin{equation*}
		\langle \mathbb{C}(\bm{g}\otimes \bm{n})\delta_S,\nabla\bm{\eta}\rangle=\int_{S}(\mathbb{C}(\bm{g}\otimes \bm{n})):\nabla\bm{\eta}\, d\sigma(\bm{x}).
	\end{equation*}
From \eqref{eq: limit_rapr}, \eqref{eq: uk_rapr}, \eqref{eq: est_rapr} and \eqref{eq: rapr_rough}, taking the limit $r\to 0$,  we obtain  the representation formula \eqref{eq: repr_u_var_coeff} exploiting  the symmetries of the tensor $\mathbb{C}$. 	
\end{proof}

\begin{remark}
 It is possible to show that $\bm{u}\in H^\sigma_{\text{loc}}(\mathbb{R}^3_-\setminus \overline{S})$, with 
 $\sigma<1$, when $\bg\in H^{1/2}(S)$. In fact, we can assume that $S$ is part of the boundary of a compact 
 Lipschitz domain $D$ with $\overline{D}\subset \RR^3_-$. Since $\partial S$ is assumed also Lipschitz, we can extend $\bm{g}$ by zero to the 
 complement of  $S$ in $\partial D$ and have that  $\bm{g}\in H^s(\partial D)$ for $s<1/2$. Then, from regularity 
 results for layer  potentials  in the case of a Lipschitz surface (see \cite[Theorem 8.7]{Mitrea-Taylor}), it follows that 
 $\bu\in H^{1/2+s}(D)$,  and $\bu\in H^{1/2+s}(\widetilde{D})$, where $\widetilde{D}$ is a compact Lipschitz domain, the boundary 
 of which contains $\partial  D$ and which is contained in the complement of $\overline D$ in $\RR^3_-$. 
\end{remark}

\subsubsection*{{\bf An explicit example.}}

We  consider now the particular case of the an isotropic homogeneous half space. We denote the constant elasticity 
tensor with $\mathbb{C}_0$ and  its Lam\'e coefficients with $\mu_0$ and $\lambda_0$. These satisfy the strong 
convexity condition $\mu_0>0$ and $3\lambda_0+2\mu_0>0$. 
In this setting, taking $S$ to be a rectangular Volterra dislocation (which, we recall, means a constant displacement jump 
distribution on $S$), in \cite{Okada} Okada  gives an explicit expression of the solution to Problem \eqref{eq: Pu}, 
highlighting the presence of singularities on the vertices of the rectangular dislocation surface. This solution is well known and 
applied in the geophysical literature (see for example \cite{Segall10,Zwieten-Hanssen-Gutierrez} and references 
therein).
   
We denote the Neumann function for a homogeneous and isotropic half space by $\mathbf{N}_0(\bm{x},
\bm{y})$. Its explicit expression can be found, for instance, in 
\cite{Mindlin36,Mindlin54,Aspri-Beretta-Mascia,Martin-Paivarinta-Rempel}. Moreover, we assume $S$ is a 
rectangle parallel to the plane $\{x_3=0\}$, that is,  
	\begin{equation}
		S=\{(y_1,y_2,y_3)\in\mathbb{R}^3:\, a\leq y_1\leq b,\, c\leq y_2\leq d,\, y_3=-|\alpha|\},
	\end{equation}
with $a, b, c, d,\alpha\in\mathbb{R}$, and we assume that $\bm{g}:=\bm{g}_c=(k_1,k_2,k_3)^T$ on $S$, with 
$k_i\in\mathbb{R}$, $i=1,2,3$. This choice for $S$ and $\bm{g}_c$ can be seen as a particular case of the 
rectangular dislocation surface considered by Okada in \cite{Okada}, and it  is the simplest case in which the source has support 
on the whole of $\overline{S}$.

In this setting, we  show that the solution $\bm{u}$ of \eqref{eq: Pu} is not in $H^1_0(\mathbb{R}^3_-\setminus 
\overline{S})$. 
The representation formula \eqref{eq: repr_u_var_coeff} becomes  
\begin{equation}\label{eq: repr_u}
         \bm{u}(\bm{x})=-\int\limits_{S}\left(\mathbb{C}_0\widehat{\nabla}_{\bm{y}}(\mathbf{N}_0 (\bm{y},
         \bm{x})) \bm{e}_3\right)^T \bm{g}_c\, d\sigma(\bm{y}).
\end{equation}
Again, we use that  $\mathbf{N}_0=\mathbf{\Gamma}+\mathbf{R}$, where $\mathbf{\Gamma}$ is the 
fundamental solution of the constant-coefficient Lam\'e operator, also known as the {\em Kelvin fundamental solution} (see \cite{Kupradze}),  and 
$\mathbf{R}$ is a regular function in $\mathbb{R}^3_-$ (see e.g. 
\cite{Aspri-Beretta-Mascia}). The singularities of $\bm{u}$ are then contained in the term  
	\begin{equation*}
		\bm{u}_{{\Gamma}}:=\int\limits_{S}(\mathbb{C}_0\widehat{\nabla}_{\bm{y}}\mathbf{\Gamma}
		(\bm{x}-\bm{y})\bm{e}_3)^T\bm{g}_c\, d\sigma(\bm{y}).
	\end{equation*} 
From straightforward calculations we find that 
	\begin{equation*}
		\mathbf{\Xi}:=(\mathbb{C}_0\widehat{\nabla}\mathbf{\Gamma}\bm{e}_3)^T=
			\begin{bmatrix}
				\mu\left(\frac{\partial \Gamma_{11}}{\partial y_3} + \frac{\partial \Gamma_{31}}{\partial 
				y_1}\right) & \mu\left(\frac{\partial \Gamma_{21}}{\partial y_3} + \frac{\partial \Gamma_{31}}
				{\partial y_2}\right) & \lambda \textrm{div}\, \bm{\varGamma}^{(1)} + 2\mu \frac{\partial 
				\Gamma_{31}}{\partial y_3}\\
				\mu\left(\frac{\partial \Gamma_{12}}{\partial y_3} + \frac{\partial \Gamma_{32}}{\partial 
				y_1}\right) & \mu\left(\frac{\partial \Gamma_{22}}{\partial y_3} + \frac{\partial \Gamma_{32}}
				{\partial y_2}\right) & \lambda \textrm{div}\, \bm{\varGamma}^{(2)} + 2\mu \frac{\partial 
				\Gamma_{32}}{\partial y_3}\\
				\mu\left(\frac{\partial \Gamma_{13}}{\partial y_3} + \frac{\partial \Gamma_{33}}{\partial 
				y_1}\right) & \mu\left(\frac{\partial \Gamma_{23}}{\partial y_3} + \frac{\partial \Gamma_{33}}
				{\partial y_2}\right) & \lambda \textrm{div}\, \bm{\varGamma}^{(3)} + 2\mu \frac{\partial 
				\Gamma_{33}}{\partial y_3}\\				
			\end{bmatrix} ,
	\end{equation*}
where $\bm{\varGamma}^{(i)}$, for $i=1,2,3$, represents the $i$-th column vector of the matrix 
$\mathbf{\Gamma}$. 
Therefore
	\begin{equation}\label{eq: u_Gamma}
		\bm{u}_{\Gamma}=\int\limits_{a}^{b}\int\limits_{c}^{d}\mathbf{\Xi}(\bm{x}-\bm{y})\bm{g}_c\, 
		d\sigma(\bm{y})=\sum_{j=1}^{3}\int\limits_{a}^{b}\int\limits_{c}^{d}\Xi_{ij}(\bm{x}-\bm{y})k_j\, 
		d\sigma(\bm{y}),\qquad i=1,2,3.
	\end{equation} 
We find that $\bm{u}_{{\Gamma}}(\bx)$ has logarithmic singularities as $\bx$ approaches  the vertices of the 
rectangle $S$, which comes exclusively from the entries $\Xi_{31}, \Xi_{32}$ (and $\Xi_{13}, \Xi_{23}$ due to the 
symmetries of $\Gamma$) of the matrix $\mathbf{\Xi}$. (See Appendix \ref{app: sing_rect} for explicit formulas of 
these terms.)

\begin{remark}
Because of the nature of the entries in $\mathbf{\Xi}$ from which the singularities originates, we stress that these singularities are present both for tangential as well as normal (and hence oblique) jumps $\bg$.
\end{remark}


\section{Formulation of the inverse problem: a uniqueness result}\label{sec: inverse problem}
In this section, we formulate and study an inverse dislocation problem. To be more precise, we are interested in the 
following boundary inverse problem:\\
\textit{ Let $\bm{u}$ be  solution to Problem \eqref{eq: Pu} with unknown dislocation surface $S$ and jump distribution $\bg$ 
	on $S$. Assume that $\bu=\bm{u}_{\Sigma}$ is known  in a bounded open region $\Sigma$ of the boundary 
	$x_3=0$, then  determine both $S$ and $\bm{g}$}. 

In particular, we want to establish under which assumptions the displacement $\bm{u}_{\Sigma}$ determines 
$S$ and $\bm{g}$ uniquely. To prove uniqueness,  we will need an extra geometrical assumption on $S$ and 
supplementary assumptions on $\bm{g}$, but we allow the elastic coefficients to be only Lipschitz continuous by adapting the proofs in \cite{Alessandrini-Rondi-Rosset-Vessella} (valid for scalar equations) to the Lam\'e system.  
Generically, we expect this regularity to be optimal for the inverse problem, unless further assumptions on the form 
of the coefficients is made, such as assuming the Lam\'e parameters to be piece-wise constant on  a given mesh. We 
reserve to address this point in future work.  We remark again that the unique continuation property may not hold  for a second-order elliptic operator, if its coefficients are only in H\"older classes $C^{0,\alpha}$, $\alpha<1$, and not Lipschitz continuous (see \cite{Miller,Plis}). \\
{To prove that the dislocation surface $S$ is uniquely determined by the boundary data, we restrict $S$ to the class $\mathcal{S}$ of open, bounded, Lipschitz, piecewise-linear surfaces that are graphs with respect to a {\em fixed}, but {\em arbitrary}, coordinate frame. This assumption is not overly restrictive in the case of faults as they are frequently nearly horizontal. Note that we do not specify the frame, so vertical faults are allowed too.}

Our uniqueness result is the following theorem.

\begin{theorem}\label{th: inv_uniq}
	{Let $S_1, S_2\in \mathcal{S}$} such that $\overline{S}_1, \overline{S}_2 \subset
	\mathbb{R}^3_-$, and  let $\bm{g}_i$ be bounded tangential fields in $ H^{1/2}(S_i)$ with \textup{supp}$\,
	\bm{g}_i=\overline{S}_i$, for $i=1,2$. Let $\bm{u}_i$, for $i=1,2$, be the unique solution of \eqref{eq: Pu} in 
	$H^{1/2-\varepsilon}_{-1/2-\varepsilon}(\mathbb{R}^3_-)$ corresponding to $\bm{g}=\bm{g}_i$ and $S=S_i$. 
	Let $\Sigma$ be an open set in $\{x_3=0\}$. If $\bm{u}_{1\big|_{\Sigma}}=\bm{u}_{2\big|
		_{\Sigma}}$,  then $S_1=S_2$ and $\bm{g}_1=\bm{g}_2$.
\end{theorem}

\begin{remark}
	We first note that by local regularity results, since $\mathbb{C}\in C^{0,1}(\overline{\mathbb{R}^3_-})$, the 
	displacements $\bm{u}_i$, for $i=1,2$, are regular in a neighborhood $\mathcal{U}$ of $\{x_3=0\}$, more precisely $u_i\in  C^{1,\alpha}(\mathcal{U})$ with $0\leq\alpha<1$ (see e.g. \cite{Li-Nirenberg}), so that $\bm{u}_{i\big|_{\Sigma}}$ can be interpreted as the trace of $\bm{u}_i$ pointwise. The same  conclusion holds for the 
	traction-free boundary condition $(\mathbb{C}\widehat{\nabla}\bm{u}_i)\bm{e}_3=0$ on $\mathbb{R}^2$.
\end{remark}

{We denote by $G$ the unbounded connected component of $\mathbb{R}^3_-\setminus \overline{S_1\cup S_2}$ containing $\Sigma$.
Since $S_1$ and $S_2$ are bounded, there exists only one such components. By definition,  $G\subseteq(\mathbb{R}^3_-\setminus \overline{S_1\cup S_2})$.
In addition, we define
\begin{equation}\label{eq: mathcal_G}
\mathcal{G}:=\partial G \setminus \{x_3=0\}.
\end{equation}
Before proving Theorem \ref*{th: inv_uniq}, we need an auxiliary result.
\begin{lemma}\label{lem: surf}
	Let $S_1, S_2\in \mathcal{S}$ such that $\overline{S}_1, \overline{S}_2 \subset \mathbb{R}^3_-$. Then $\mathcal{G}=\overline{S_1\cup S_2}$.	
\end{lemma}	
\begin{proof}
	We begin by observing that, if $G=\mathbb{R}^3_-\setminus \overline{S_1\cup S_2}$, then clearly  $\mathcal{G}=\overline{S_1\cup S_2}$.
	If $G\subset (\mathbb{R}^3_-\setminus \overline{S_1\cup S_2})$, then by definition $\mathcal{G}\subseteq \overline{S_1\cup S_2}$.
	We prove the reverse inclusion, i.e., $\overline{S_1\cup S_2}\subseteq \mathcal{G}$, arguing by contradiction.
	We hence assume that there exists a point $\bm{x}\in\overline{S_1\cup S_2}$ (for instance, without loss of generality, $\bm{x}\in\overline{S}_1$) such that $\bm{x}\notin\mathcal{G}$. Given  any vector $\bm{v}\neq \bm{0}$, we consider the line $r$ through $\bm{x}$ parallel to $\bm{v}$.
	Since $\mathcal{G}$ is closed and  $\bm{x}\notin\mathcal{G}$, line $r$ must intersecate $\mathcal{G}$ in at least two points
	$\bm{x_1}$, $\bm{x}_2$.
	Therefore, $\bm{x}, \bm{x}_1, \bm{x}_2\in r$ and $\bm{x}_1, \bm{x}_2\in\mathcal{G}$. Now, choosing $\bm{v}$ in the direction along which $S_1$ and $S_2$ are graphs, we have that
	\begin{enumerate}
		\item either $\bm{x}_1$ or $\bm{x}_2$ belongs to $S_1$, hence we reach a contradiction, given that $S_1$ is a graph.
		\item $\bm{x}_1$ and $\bm{x}_2$ belong to $S_2$, which is also a contradiction, since $S_2$ is a graph.
	\end{enumerate}
	Therefore $\mathcal{G}=\overline{S_1\cup S_2}$.	
\end{proof}
}
{We are now in the position to prove the uniqueness theorem.}

\begin{proof}[Proof of Theorem \ref{th: inv_uniq}]
	We let $\bm{w}=\bm{u}_1-\bm{u}_2$, defined on  $\mathbb{R}^3_-\setminus \overline{S_1\cup S_2}$. Then 
	\begin{equation*}
	\bm{w}_{|_{\Sigma}}=(\mathbb{C}\widehat{\nabla}\bm{w})\bm{e}_3 {_{|_{\Sigma}}}=\bm{0},\qquad 
	\textrm{and}\,\qquad \textrm{div}(\mathbb{C}\widehat{\nabla}\bm{w})=\bm{0}\,\,\, \textrm{in}\,\,\, 
	\mathbb{R}^3_-\setminus \overline{S_1\cup S_2}.
	\end{equation*}
	From the homogeneous boundary conditions, it follows that also the Cauchy data are zero. Hence,
	\begin{equation}
	\bm{w}_{|_{\Sigma}}=\frac{\partial\bm{w}}{\partial x_3}_{\Big|_{\Sigma}}=\bm{0}.
	\end{equation}	
	Without loss of generality, we assume that $\Sigma=B'_R(\bm{0})$  (see Section \ref{sec: notation and functional setting} for notation).
	We define  $\widetilde{\bm{w}}$, extension of $\bm{w}$, in $B_R(\bm{0})$ by
	\begin{equation}
	\widetilde{\bm{w}}=
	\begin{cases}
	\bm{w}, & \textrm{in}\, B^-_{R}(\bm{0})\cup \Sigma,\\
	\bm{0}, 	   & \textrm{in}\, B^+_{R}(\bm{0}).
	\end{cases}
	\end{equation}
	{Clearly by the continuity property of $\widetilde{\bm{w}}$ we have that $\widetilde{\bm{w}}\in H^1(B_R(\bm{0}))$.}
	We also extend the elasticity tensor $\CC$ to $B^+_{R}(\bm{0})$ in such a way that the resulting tensor is Lipschitz 
	continuous and strongly convex in $B_{R}(\bm{0})$ (for example, by even reflection).
	We will show that $\widetilde{\bm{w}}$ is a weak solution to the Lam\'e system in $B_{R}(\bm{0})$. We let
	$\bm{\varphi}\in H^1_0(B_R(\bm{0}))$, and test the equation with $\bm{\varphi}$:
	\begin{equation*}
	\int_{B_{R}(\bm{0})}\text{div}( \mathbb{C}\widehat{\nabla}\widetilde{\bm{w}})\cdot \bm{\varphi}\, 
	d\bm{x}=-\int_{B^-_R(\bm{0})} (\mathbb{C}\widehat{\nabla}\bm{w}):\widehat{\nabla}\bm{\varphi}\,
	d\bm{x}=0,
	\end{equation*}
	where in the last equality we have integrated by parts, and used the fact that $\bm{w}$ satisfies the Lam\'e system in $B^-_R(\bm{0})$.
	Consequently, $\widetilde{\bm{w}}\in H^1(B_{R}(\bm{0}))$ is a {weak solution} to
	\begin{equation*}
	\textrm{div}(\mathbb{C}\widehat{\nabla}\widetilde{\bm{w}})=\bm{0},\qquad \textrm{in}\, 
	B_{R}(\bm{0}).
	\end{equation*} 
	We can then apply  Theorem 1.3 in \cite{Lin-Nakamura-Wang} to obtain the strong unique continuation property in 
	$B_{R}(\bm{0})$, which implies the weak continuation property. Since $\widetilde{\bm{w}}=\bm{0}$ in 
	$B^+_{R}(\bm{0})$ and the weak continuation property holds in $B_{R}(\bm{0})$, we have that $\widetilde{\bm{w}}=\bm{0}$ in $B_{R}(\bm{0})$. That is, 
	$\widetilde{\bm{w}}=\bm{w}=\bm{0}$ in $B^-_{R}(\bm{0})$. 
	Using the three-spheres inequality (see Theorem 1.1 in \cite{Lin-Nakamura-Wang}) in the connected component 
	$G$ of $\mathbb{R}^3_-\setminus \overline{S_1 \cup S_2}$ containing $\Sigma$ gives  that $\bm{w}=\bm{0}$ in $G$. We  now distinguish two cases:
	\begin{enumerate}[label=(\roman*), ref=(\roman*)]
		\item {$G=\mathbb{R}^3_-\setminus\overline{S_1\cup S_2}$;}\label{i.unique1}
		\item {$G\subset \mathbb{R}^3_-\setminus\overline{S_1\cup S_2}$.} \label{i.unique2}
	\end{enumerate}	
	
	We start from Case \ref{i.unique1}, {see Figure \ref{fig. inv1} for an example of the geometrical setting}. 
	\begin{figure}[h!]
	\centering
	\includegraphics[angle=25,scale=0.5]{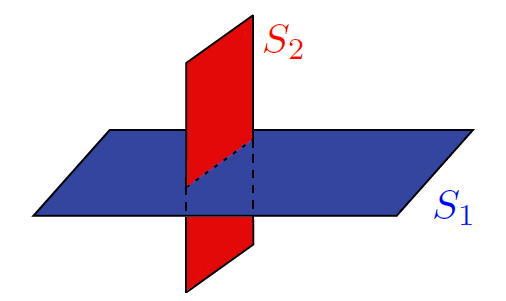}
	\caption{An example of the geometrical setting in Case \ref{i.unique1}.}\label{fig. inv1}
	\end{figure}
	We assume that $\overline{S}_1\neq \overline{S}_2$, and we  fix $\bm{y}\in 
	\overline{S}_1$ such that $\bm{y}\notin\overline{S}_2$. Then there exists a ball $B_r(\bm{y})$ that does not
	intersect $\overline{S}_2$. Hence,
	\begin{equation*}
	\bm{0}=[\bm{w}]_{B_r(\bm{y})\cap S_1}=[\bm{u}_1]_{B_r(\bm{y})\cap S_1}=\bm{g}_1,
	\end{equation*}	
	and this equality {leads to a contradiction, as $\text{supp}\,\bm{g}_1=\overline{S}_1$}. We can repeat the same
	argument, switching the role of $S_1$ and $S_2$, to conclude that $\overline{S}_1=\overline{S}_2$.
	Therefore, 
	\begin{equation*}
	\bm{0}=[\bm{w}]_{S_1}=[\bm{w}]_{S_2}\Rightarrow [\bm{u}_1]_{S_1}=[\bm{u}_2]_{S_2}\Rightarrow 
	\bm{g}_1=\bm{g}_2. 
	\end{equation*}
	We  now turn to Case  \ref{i.unique2}, {see Figure \ref{fig. inv2} for an example of the geometrical setting.}
	\begin{figure}[h!]
		\centering
		\includegraphics{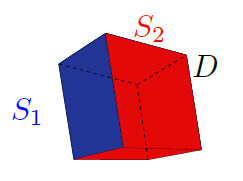}
		\caption{An example of the domain $D$ in Case \ref{i.unique2} where the two surfaces close.}\label{fig. inv2}
	\end{figure}
	{By Lemma \ref{lem: surf}, we can assume, without loss of generality, that there exists a bounded connected domain $D$ such that $\partial D=\overline{S_1\cup S_2}$. Then
	$\bm{w}=\bm{0}$ in a neighborhood of $\partial D$ in $D^C\cap \mathbb{R}^3_-$, since $\bm{w}=\bm{0}$ in $G$.
	The continuity of the tractions $(\mathbb{C}\widehat{\nabla}\bm{u}_1)\bm{n}$ and $(\mathbb{C}\widehat{\nabla}\bm{u}_2)\bm{n}$ in trace sense} across $S_1$ and $S_2$, respectively, implies that
	\begin{equation}\label{eq: conor_w+}
	(\mathbb{C}\widehat{\nabla}\bm{w}^+)\bm{n}=\bm{0}\, \quad \textrm{a.e. on}\,\, \partial D,
	\end{equation}
	where $\bm{w}^+$ indicates the function $\bm{w}$ restricted to $D$ and $\bm{n}$ the outward unit normal to $D$. Moreover, $\bm{w}^+$ satisfies
	\begin{equation}\label{eq: equat_w+}
	\textrm{div}(\mathbb{C}\widehat{\nabla}\bm{w}^+)=\bm{0}\,\quad \textrm{in}\,\, D.
	\end{equation} 
	We conclude  from \eqref{eq: conor_w+} and \eqref{eq: equat_w+} that  $\bm{w}^+$ is a rigid motion, i.e.,
	\begin{equation*}
	\bm{w}^+=\mathbf{A}\bm{x}+\bm{c},
	\end{equation*}	
	where $\bm{c}\in\mathbb{R}^3$ and $\mathbf{A}\in\mathbb{R}^{3\times3}$ is a skew matrix.\\ 
	On the other hand, since $\bm{w}^-=\bm{0}$ on $\partial D$, we have 
	\begin{equation*}
	[\bm{w}]_{\partial D}=\bm{w}^+_{|_{\partial D}},
	\end{equation*}	
	but on $S_1$ we have $[\bm{w}]_{S_1}=[\bm{u}_1]_{S_1}$, which is tangential by assumption; the same argument can be repeated on $S_2$. Therefore,
	\begin{equation}\label{eq: wdotn}
	\bm{w}^+\cdot\bm{n} =\bm{0}\, \qquad \textrm{on}\,\, \partial D.
	\end{equation}
	We will now show that \eqref{eq: wdotn} implies that $\mathbf{A}$ and $\bm{c}$ are zero, which means that $\bm{w}=\bm{0}$ in $\mathbb{R}^3_-\setminus \overline{S_1\cup S_2}$. Then,  reasoning as for Case \ref{i.unique1},  $\overline{S}_1=\overline{S}_2$ and $\bm{g}_1=\bm{g}_2$. 
	Let  $\bm{x}_0$ denote a vertex of $\overline{S_1\cup S_2}$. This vertex  is contained in at least three faces  $F_1, F_2, F_3$ having independent normals $\bm{n}_1,\bm{n}_2,\bm{n}_3$.
	We pick  three points $\bm{x}_1, \bm{x}_2, \bm{x}_3$ on the faces $F_1, F_2, F_3$. Then
	\begin{equation*}
	(\mathbf{A}\bm{x}_i+\bm{c})\cdot\bm{n}_i={0},\,\qquad \forall\, i=1,2,3.
	\end{equation*}
	By continuity, there is a well-defined limit of $(\mathbf{A}\bm{x}+\bm{c})\cdot\bm{n}$ when $\bm{x}$ approaches $\bm{x}_0$. It follows that
	\begin{equation*}
	(\mathbf{A}\bm{x}_0+\bm{c})\cdot \bm{n}_i={0},\qquad \forall\, i=1,2,3,
	\end{equation*}
	which implies by linear independence that 
	\begin{equation}\label{eq: Ax0+c}
	\mathbf{A}\bm{x}_0+\bm{c}=\bm{0}. 
	\end{equation}	
	An analogous  argument can be applied to obtain a similar relation for any other vertex in $\overline{S_1\cup S_2}$. Now, observe
	that there exist at least three vertices that are not co-planar.  We denote them by $\bm{x}^1_0,\bm{x}^2_0, 
	\bm{x}^3_0$. We assume that $\bm{x}^1_0$ does not lie in the plane spanned by $\bm{x}^2_0, 
	\bm{x}^3_0$. Then the vectors $\bm{x}^1_0-\bm{x}^2_0$ and  $\bm{x}^1_0-\bm{x}^3_0$ are independent. Combining  by linearity all the relations of the form 
	\eqref{eq: Ax0+c} for the three vertices gives:
	\begin{equation*}
	\begin{cases}
	\mathbf{A}(\bm{x}^1_0-\bm{x}^2_0)=\bm{0},\\
	\mathbf{A}(\bm{x}^1_0-\bm{x}^3_0)=\bm{0}.\\
	\end{cases}
	\end{equation*}
	The above relations give six independent scalar linear equations for the entries of $A$.
	Since $\mathbf{A}$ is a skew-symmetric  matrix, it follows that $\mathbf{A}=\mathbf{0}$ and, therefore, $\bm{c}=\bm{0}$. 		
\end{proof}
{\begin{remark}
The geometrical assumption on the faults, which requires that the two surfaces are graphs with respect to a fixed, but arbitrary, coordinate frame, is needed in the second part of the proof (case \ref{i.unique2}) to avoid the exceptional case where the two surfaces, or one of the two, could have a component inside the closed domain $D$, see Figure \ref{Fig: sup_cl_comp} to visualize an example of the situation we are describing (the dashed part is inside $D$).
			\begin{figure}[h!]
				\centering
				\includegraphics[scale=0.5]{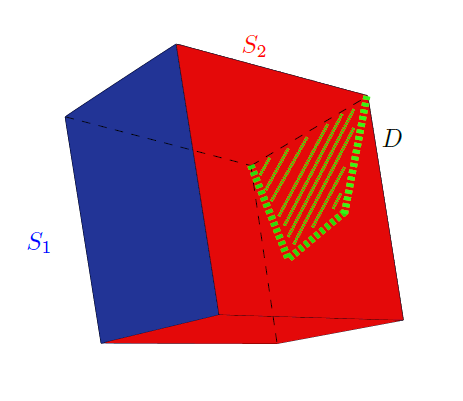}
				\caption{Component inside $D$.}\label{Fig: sup_cl_comp}
			\end{figure} 
			 This hypothesis is sufficient, but certainly not necessary. However, without this assumption, a situation like that one in Figure \ref{Fig: sup_cl_comp} could appear and in this case our argument wouldn't apply. In fact, we would have less information on $\bm{w}$ in the disconnected component, namely, we would not know that the traction is zero on this component but only that it is continous across it. 
\end{remark} }
{\begin{remark}
The first part of the proof (case \ref{i.unique1}) still holds without any additional assumptions on the fault surfaces $S_i$, and slip vectors $\bm{g}_i$, $i=1,2$, beyond those required for the well-posedness of the forward problem and the validity of unique continuation. However, our approach does not allow us to treat the second case (case \ref{i.unique2}). We observe, on the other hand, that even in this second case, uniqueness holds up to an infinitesimal rigid motion, under the geometric assumptions on the fault we make (we still need to exclude {the special case of Figure \ref{Fig: sup_cl_comp}}). Therefore, generically with respect to the slip vectors, we can prove uniqueness, since the space of infinitesimal rigid motions is finite dimensional, while the space of slip vectors $\bm{g}$ is infinite dimensional. Unfortunately, this case cannot be excluded {\em a priori} in applications.
\end{remark}
}		
\begin{remark}
We observe that the statement of the previous theorem is true also in the case of slip fields directed in the normal
		direction. Indeed, instead of  \eqref{eq: wdotn}, we have $\, \bm{w}^+\cdot\bm{\tau} =\bm{0}$ on $\partial D$, where $\bm{\tau}$ is any tangent vector to $\partial  D$, so that we can still derive three independent conditions at each vertex and repeat the same arguments of the last part of the proof. In the oblique case, this condition cannot be guaranteed.
 
 		The condition that the surface $S$ is piece-wise linear is mathematical restrictive. However, it includes the important case in applications of surfaces that are exactly triangularized by polyhedral meshes. In addition, we only use the fact that  piece-wise linear, closed surfaces have at least three vertices that are not co-planar and each vertex is adjoined by three faces with  normal vectors that are independent of each other. Consequently, we can extend our result to more general surfaces that satisfy these conditions. For example, we can consider surfaces that are exactly triangularized by curvilinear meshes, as long as each element of the mesh is not co-planar with any other element, so that vertices of adjoining elements form real corners.  
 
 		The significance of effects originating from  heterogeneity around and across the fault (including damage zones)  and the fault geometric complexity has been recognized since the mid  1980s. These are still subject to active research; for a recent  study, see \cite{ZielkeMai}. Fault geometric complexities have been shown to strongly affect a                 
 		fault’s seismic behavior at the corresponding spatial scales, while  affecting the dynamic rupture process.    
\end{remark}

\appendix 
\begin{appendices}

\section{Singularities for a rectangular dislocation surface parallel to $\{x_3=0\}$}\label{app: sing_rect}
In this section we give the explicit expression for $\bu_\Gamma$ defined in  \eqref{eq: u_Gamma}. To this end, we define the following functions:
	\begin{equation*}
		\begin{aligned}
			h_0(x,y,z)&=\sqrt{x^2+y^2+z^2},\qquad\qquad\qquad
			h_1(x,y,z)=\arctan\left(\frac{xy}{z h_0(x,y,z)}\right),\\
			h_2(x,y,z)&=\frac{xyz}{(x^2+z^2)h_0(x,y,z)},\qquad\qquad
			h_3(x,y,z)=\ln(y+h_0(x,y,z)),\\
			h_4(x,y,z)&=\frac{yz^2}{(x^2+z^2)h_0(x,y,z)},\qquad\qquad
			h_5(x,y,z)=\frac{xyz h_0(x,y,z)}{(x^2+z^2)(y^2+z^2)},\\
			h_6(x,y,z)&=\frac{xyz^3}{(x^2+z^2)(y^2+z^2)h_0(x,y,z)}
		\end{aligned}
	\end{equation*}
We denote $x^1_a:=x_1-a$, $x^1_b:=x_1-b$, $x^2_c:=x_2-c$, $x^2_d:=x_2-d$, $x^3_{\alpha}:=x_3+|\alpha|$, and define the constant $C_{\nu}=-1/(8\pi(1-\nu))$. 
We calculate  the integrals in \eqref{eq: u_Gamma} (some of the calculations were performed using the computer algebra software Mathematica$^{\text{\copyright}}$), and  we find that
	\begin{equation*}
		\begin{aligned}
			C_{\nu}^{-1}(u_{\Gamma})_1=\frac{2k_1(\lambda+2\mu)}{\lambda+\mu}&\left[h_1(x^1_a,x^2_c,x^3_{\alpha})-h_1(x^1_a,x^2_d,x^3_{\alpha})+h_1(x^1_b,x^2_d,x^3_{\alpha})-h_1(x^1_b,x^2_c,x^3_{\alpha})\right]\\
			+k_1&\left[h_2(x^1_a,x^2_d,x^3_{\alpha})-h_2(x^1_a,x^2_c,x^3_{\alpha})+h_2(x^1_b,x^2_c,x^3_{\alpha})-h_2(x^1_b,x^2_d,x^3_{\alpha})\right]\\
			+k_2 x^3_{\alpha}&\left[\frac{1}{h_0(x^1_a,x^2_c,x^3_{\alpha})}-\frac{1}{h_0(x^1_b,x^2_c,x^3_{\alpha})}+\frac{1}{h_0(x^1_b,x^2_d,x^3_{\alpha})}-\frac{1}{h_0(x^1_a,x^2_d,x^3_{\alpha})}\right]\\
			+\frac{k_3\mu}{\lambda+\mu}&\left[h_3(x^1_a,x^2_c,x^3_{\alpha})-h_3(x^1_b,x^2_c,x^3_{\alpha})+h_3(x^1_b,x^2_d,x^3_{\alpha})-h_3(x^1_a,x^2_d,x^3_{\alpha})\right]\\
			+k_3&\left[h_4(x^1_a,x^2_d,x^3_{\alpha})-h_4(x^1_a,x^2_c,x^3_{\alpha})+h_4(x^1_b,x^2_c,x^3_{\alpha})-h_4(x^1_b,x^2_d,x^3_{\alpha})\right],
		\end{aligned}
	\end{equation*}
	\begin{equation*}
			\begin{aligned}
				C_{\nu}^{-1}(u_{\Gamma})_2={k_1x^3_{\alpha}}&\left[\frac{1}{h_0(x^1_a,x^2_c,x^3_{\alpha})}-\frac{1}{h_0(x^1_b,x^2_c,x^3_{\alpha})}+\frac{1}{h_0(x^1_b,x^2_d,x^3_{\alpha})}-\frac{1}{h_0(x^1_a,x^2_d,x^3_{\alpha})}\right]\\
				+\frac{2k_2(\lambda+2\mu)}{\lambda+\mu}&\left[h_1(x^1_a,x^2_c,x^3_{\alpha})-h_1(x^1_a,x^2_d,x^3_{\alpha})+h_1(x^1_b,x^2_d,x^3_{\alpha})-h_1(x^1_b,x^2_c,x^3_{\alpha})\right]\\
				+k_2&\left[h_2(x^2_d,x^1_a,x^3_{\alpha})-h_2(x^2_c,x^1_a,x^3_{\alpha})+h_2(x^2_c,x^1_b,x^3_{\alpha})-h_2(x^2_d,x^1_b,x^3_{\alpha})\right]\\
				+\frac{k_3\mu}{\lambda+\mu}&\left[h_3(x^2_c,x^1_a,x^3_{\alpha})-h_3(x^2_d,x^1_a,x^3_{\alpha})+h_3(x^2_d,x^1_b,x^3_{\alpha})-h_3(x^2_c,x^1_b,x^3_{\alpha})\right]\\
				+k_3&\left[h_4(x^2_d,x^1_a,x^3_{\alpha})-h_4(x^2_c,x^1_a,x^3_{\alpha})+h_4(x^2_c,x^1_b,x^3_{\alpha})-h_4(x^2_d,x^1_b,x^3_{\alpha}) \right] ,			
			\end{aligned}
		\end{equation*}
		\begin{equation*}
					\begin{aligned}
						C_{\nu}^{-1}(u_{\Gamma})_3=\frac{k_1\mu}{\lambda+\mu}&\left[h_3(x^1_b,x^2_c,x^3_{\alpha})-h_3(x^1_a,x^2_c,x^3_{\alpha})+h_3(x^1_a,x^2_d,x^3_{\alpha})-h_3(x^1_b,x^2_d,x^3_{\alpha})\right]\\
						+k_1&\left[h_4(x^1_a,x^2_d,x^3_{\alpha})-h_4(x^1_a,x^2_c,x^3_{\alpha})+h_4(x^1_b,x^2_c,x^3_{\alpha})-h_4(x^1_b,x^2_d,x^3_{\alpha})\right]\\
						+\frac{k_2\mu}{(\lambda+\mu)}&\left[h_3(x^2_d,x^1_a,x^3_{\alpha})-h_3(x^2_c,x^1_a,x^3_{\alpha})+h_3(x^2_c,x^1_b,x^3_{\alpha})-h_3(x^2_d,x^1_b,x^3_{\alpha})\right]\\
						+k_2&\left[h_4(x^2_d,x^1_a,x^3_{\alpha})-h_4(x^2_c,x^1_a,x^3_{\alpha})+h_4(x^2_c,x^1_b,x^3_{\alpha})-h_4(x^2_d,x^1_b,x^3_{\alpha})\right]\\
						+\frac{k_3(\lambda+2\mu)}{\lambda+\mu}&\left[h_1(x^1_a,x^2_c,x^3_{\alpha})-h_1(x^1_a,x^2_d,x^3_{\alpha})+h_1(x^1_b,x^2_d,x^3_{\alpha})-h_1(x^1_b,x^2_c,x^3_{\alpha})\right]\\
						+k_3&\left[h_5(x^1_a,x^2_c,x^3_{\alpha})-h_5(x^1_a,x^2_d,x^3_{\alpha})+h_5(x^1_b,x^2_d,x^3_{\alpha})-h_5(x^1_b,x^2_c,x^3_{\alpha})\right]\\
						+k_3&\left[h_6(x^1_a,x^2_c,x^3_{\alpha})-h_6(x^1_a,x^2_d,x^3_{\alpha})+h_6(x^1_b,x^2_d,x^3_{\alpha})-h_6(x^1_b,x^2_c,x^3_{\alpha})\right] .
					\end{aligned}
				\end{equation*}

\end{appendices}


\begin{thebibliography}{99}

\bibitem{Alessandrini-Rondi-Rosset-Vessella} Alessandrini G., Rondi L., Rosset E., Vessella S.,
The stability for the Cauchy problem for elliptic equations.
{\it Inverse Problems} {\bf 25} 123004, 47pp (2009).

\bibitem{Amrouche-Bonzom} Amrouche C., Bonzom F.,
Exterior problems in the half-space for the Laplace operator in weighted Sobolev spaces.
{\it Journal of Differential Equations} {\bf 246 }, 1894--1920 (2009).

\bibitem{Amrouche-Dambrine-Raudin} Amrouche C., Dambrine M., Raudin Y., 
An $L^p$ theory of linear elasticity in the half-space.
{\it Journal of Differential Equations} {\bf 253}, 906--932 (2012).


\bibitem{Amrouche-Girault-Giroire} Amrouche C., Girault V., Giroire J.,
Weighted Sobolev spaces for Laplace's equation in $\mathbb{R}^n$.
{\it Journal de Math\'ematiques Pures et Appliqu\'ees} {\bf 73}, 579--606 (1994).


\bibitem{Amrouche-Necasova} Amrouche C., Ne\v{c}asov\'a S.,
Laplace equation in the half-space with a nonhomogeneous Dirichlet boundary condition. 
{\it Mathematica Bohemica} {\bf 126}, 265--274 (2001).


\bibitem{Amrouche-Necasova-Raudin} Amrouche C., Ne\v{c}asov\'a S., Raudin Y.,
Very weak, generalized and strong solutions to the Stokes system in the half-space.
{\it Journal of Differential Equations} {\bf 244}, 887--915 (2008).


\bibitem{Arnodottir-Segall} \'Arnad\'ottir T., Segall P.,
The 1989 Loma Prieta earthquake imaged from inversion of geodetic data.
{\it Journal of Geophysical Research} {\bf 99}, 21,835 -- 21,855 (1994).


\bibitem{Aspri-Beretta-Mascia} Aspri A., Beretta E., Mascia C., Analysis of a Mogi-type model describing surface deformations induced by a magma chamber embedded in an elastic half-space.
{\it Journal de l'École Polytechnique — Mathématiques} {\bf 4}, 223--255 (2017). 


\bibitem{Aspri-Beretta-Rosset} Aspri A., Beretta E., Rosset E.,
On an elastic model arising from volcanology: An analysis of the direct and inverse problem.
{\it Journal of Differential Equations} {\bf 265}, 6400--6423 (2018).
 

\bibitem{Beretta-Francini-Vessella}
Beretta E., Francini E., Vessella S.,  
Determination of a linear crack in an elastic body from boundary measurements -- Lipschitz stability. 
{\it SIAM Journal of Mathematical Analysis} {\bf 40} no.3, 984--1002 (2008).
 
 
\bibitem{Beretta-Francini-Kim-Lee}
Beretta E., Francini E., Kim E, Lee J.-Y.,
Algorithm for the determination of a linear crack in an elastic body from boundary measurements.
{\it Inverse Problems} {\bf 26} no. 8, 085015 (2010).


\bibitem{Bergh-Lofstrom} Bergh J., L\"{o}fstr\"{o}m J.,
Interpolation Spaces - An Introduction.
Springer-Verlag Berlin Heidelberg, (1976).


\bibitem{Bonafede-Rivalta} Bonafede M., Rivalta E., 
The tensile dislocation problem in a layered elastic medium.
{\it Geophysical Journal International} {\bf 136}, 341--356 (1999).
{
	\bibitem{Cambiotti-Zhou-Sparacino-Sabadini-Sun}
Cambiotti G., Zhou X., Sparacino F., Sabadini R.,  Sun W.,
Joint estimate of the rupture area and slip distribution of the 2009 L’Aquila earthquake by a Bayesian inversion of GPS data. 
{\it Geophysical Journal International}, {\bf 209}, 992--1003 (2017).
}
{
\bibitem{Cohen}	
Cohen S.,
Convenient formulas for determining dip-slip fault                       
parameters from geophysical observables. 
{\it Bulletin of the Seismological Society of America}, {\bf 86}, 1642–1644 (1996).	
}	
\bibitem{ColliFranzone-Guerri-Magenes} Colli Franzone P., Guerri L., Magenes E.,
Oblique double layer potentials for the direct and inverse problems of electrocardiology.
{\it Mathematical Biosciences} {\bf 68}, 23--55 (1984).
{
	\bibitem{Deloius-Nocquet-Vallee}
	Deloius B., Nocquet J.-M., Vallée M.,
	Slip distribution of the February 27, 2010 Mw= 8.8 Maule earthquake, central Chile, from static and high‐rate GPS, InSAR, and broadband teleseismic data.
	{\it Geophysical Research Letters}, {\bf 37}, no. 17 (2010).
}	
\bibitem{Eshelby} Eshelby J. D.,
Dislocation theory for geophysical applications.
{\it Philosophical Transactions of the Royal Society A }  {\bf  274 }, 331--338 (1973).


\bibitem{Evans-Meade} Evans E. L., Meade B.J.,
Geodetic imaging of coseismic slip and postseismic afterslip:
Sparsity promoting methods applied to the great
Tohoku earthquake.
{\it Geophysical Research Letters} {\bf 39}, 1--7 (2012).


\bibitem{Fuchs84} Fuchs M., 
The Green-matrix for elliptic systems which satisfy the Legendre-Hadamard condition.
{\it Manuscripta Mathematica} {\bf 46}, 97--115 (1984).


\bibitem{Fuchs86} Fuchs M.,
The Green Matrix for Strongly Elliptic Systems of Second Order with Continuous Coefficients.
{\it Zeitschrift f\"{u}r Analysis und ihre Anwendungen} {\bf 6}, 507--531 (1986).


\bibitem{Fukahata-Wright} Fukahata Y., Wright T.J.,
A non-linear geodetic data inversion using ABIC for slip distribution
on a fault with an unknown dip angle.
{\it Geophysical Journal International} {\bf 173}, 353--364 (2008).


\bibitem{Hanouzet} Hanouzet B.,
Espaces de Sobolev avec poids. Application au probl\`eme de Dirichlet dans un demi-espace.
{\it Rendiconti del Seminario Matematico della Universit\`a di Padova} {\bf 46}, 227--272 (1971).
{
\bibitem{Jiang-Wang-Wu-Che-Shen}
Jiang Z., Wang M., Wang Y., Wu Y., Che, S., Shen Z. K., Bürgmann R., Sun J., Yang Y., Liao H., Li Q., 
GPS constrained coseismic source and slip distribution of the 2013 Mw6. 6 Lushan, China, earthquake and its tectonic implications. 
{\it Geophysical Research Letters}, {\bf 41}, 407--413 (2014).
}
{
\bibitem{Johnson}	
Johnson K. M., Hsu Y. J., Segall P, Yu S. B.,
Fault geometry and slip distribution of the 1999 Chi-Chi, Taiwan, earthquake                   
imaged from inversion of GPS data.
{\it Geophysical Research Letters}, {\bf 28}, 2285–-2288 (2001). 
}	
\bibitem{Koch-Lin-Wang} Koch H., Lin C-L, Wang J-N,
Doubling inequalities for the Lam\'e system with rough coefficients.
{\it Proceedings of the American Mathematical Society} {\bf 144}, 5309--5318 (2016).


\bibitem{Kondracev-Oleinik} Kondrat'ev V.A., Oleinik O.A.,
Boundary-value problems for the system of elasticity theory in unbounded domains. Korn's inequalities.
{\it Russian Math. Surveys} {\bf 43}, 65--119 (1988).


\bibitem{Kupradze} Kupradze V.D., Potential Methods in the Theory of Elasticity. Israel Program for Scientific Translations, Jerusalem, (1965).

\bibitem{Li-Nirenberg} Li, Y.,  Nirenberg, L.,
Estimates for elliptic systems from composite material.
Dedicated to the memory of Jürgen K. Moser. {\em Comm. Pure Appl. Math.} {\bf 56} n. 7, 892--925 (2003). 

\bibitem{Lin-Nakamura-Wang} Lin C.-L., Nakamura G., Wang J.-N.,
Optimal three-ball inequalities and quantitative uniqueness for the Lam\'e system with Lipschitz coefficients.
{\it Duke Mathematical Journal}, {\bf 155} no. 1, 198--204 (2010).


\bibitem{Lions-Magenes} Lions J.L., Magenes E.,
Non-Homogeneous Boundary Value Problems and Applications.
Volume I, Springer-Verlag Berlin Heidelberg New York, (1972).

\bibitem{Martin-Paivarinta-Rempel} Martin P.A., Päivärinta L., Rempel S.,
A normal crack in an elastic half-space with stress-free surface.
{\it Mathematical Methods in the Applied Sciences} {\bf 16}, 563--579 (1993).

\bibitem{Melrose} Melrose R. b., The Atiyah-Patodi-Singer index theorem. 
{\em Research Notes in Mathematics, \bf 4}. A K Peters, Ltd., Wellesley, MA, (1993).


\bibitem{Miller} Miller K., Nonunique continuation for uniformly parabolic and elliptic equations in
selfadjoint divergence form with H\"{o}lder continuous coefficients. 
{\it Bull. Amer. Math. Soc.} {\bf 79}, 350 -- 354 (1973).

 
\bibitem{Mindlin36} Mindlin R. D.,
 Force at a Point in the Interior of a SemiInfinite Solid.
{\it Journal of Applied Physics}, {\bf 7}, 195--202 (1936).


\bibitem{Mindlin54} Mindlin R. D., 
Force at a Point in the Interior of a Semi-Infinite Solid. 
{\it Proceedings of The First Midwestern Conference on Solid Mechanics}, April, University of Illinois, Urbana, Ill., (1954).


\bibitem{Mitrea-Taylor} Mitrea D., Mitrea M., Taylor M.,
Layer potentials, the Hodge laplacian, and global boundary problems in nonsmooth Riemannian manifolds.
{Memoirs of the American Mathematical Society} {\bf 150}, n. 713, (2001). 


\bibitem{Morassi-Rosset} Morassi A., Rosset E., 
Stable determination of cavities in elastic bodies.
{\it Inverse Problem} {\bf 20}, 453--480 (2004).
{
\bibitem{Nikkhoo-Walter}	
Nikkhoo M., Walter T. R.,
Triangular dislocation: an analytical, artefact-free solution.
{\it Geophysical Journal International}, {\bf 201}, 1119--1141 (2015).	
}

\bibitem{Okada} Okada Y., Internal deformation due to shear and tensile fault in a half-space.
{\it Bulletin of the Seismological Society of America} {\bf 82} no. 2, 1018--1040 (1992).


\bibitem{Plis} Plis A., On non-uniqueness in Cauchy problem for an elliptic second order differential
equation. 
{\it Bull. Acad. Polon. Sci. Ser. Sci. Math. Astronom. Phys.}  {\bf 11}, 95 -- 100 (1963).

 
\bibitem{Rivalta-Mangiavillano-Bonafede} Rivalta E., Mangiavillano W., Bonafede M., The edge dislocation problem in a layered elastic medium.
{\it Geophysical Journal International}, {\bf 149}, 508--523 (2002).

\bibitem{Segall10}
Segall P.,
Earthquake and volcano deformation.
Princeton University Press, (2010). 
{
\bibitem{Serpelloni-Anderlini-Belardinelli}	
Serpelloni E., Anderlini L., Belardinelli M. E., 
Fault geometry, coseismic-slip distribution and Coulomb stress change associated with the 2009 April 6, M W 6.3, L'Aquila earthquake from inversion of GPS displacements. 
{\it Geophysical Journal International}, {\bf 188}, 473--489 (2012).
}
{
	\bibitem{Simons-Fialko-Rivera} Simons M., Fialko Y., Rivera L.,
	Coseismic deformation from the 1999 M w 7.1 Hector Mine, California, earthquake as inferred from InSAR and GPS observations.
	{\it Bulletin of the Seismological Society of America} {\bf 92} no. 4, 1390--1402 (2002).
}
{
\bibitem{Trasatti-Hyriakopoulos-Chini}		
Trasatti E., Kyriakopoulos C., Chini M., 
Finite element inversion of DInSAR data from the Mw 6.3 L'Aquila earthquake, 2009 (Italy). {Geophysical Research Letters}, {\bf 38}, (2011).
}	
\bibitem{TriebelWeight}
Triebel H.,  Spaces of Kudrjavcev Type I. Interpolation, Embedding, and Structure. {\em Journal of Mathematical Analysis and Applications} {\bf 56}, 253--271 (1976).

\bibitem{Triki-Volkov}
Triki F. Volkov D., Stability estimates for the fault inverse problem.
{\em Inverse Problems} {\bf 35}, 075007 (2019).

\bibitem{Volkov-Voisin-Ionescu}
Volkov D., Voisin C., Ionescu R.,
Reconstruction of faults in elastic half space from surface measurements.
{\it Inverse Problems} {\bf 33}, 055018 (2017).


\bibitem{Volterra} Volterra V.,
Sur l'equilibre des corps elastiques multiplement connexes.
{\it Annales scientifiques de l' \'Ecole Normale Sup\'erieure} {\bf 24}, 401--517 (1907).
{
\bibitem{Walker-Bergman-Szeliga-Fielding}	
Walker R. T., Bergman E. A., Szeliga W., Fielding E. J.,	
Insights into the 1968-1997 Dasht-e-Bayaz and Zirkuh earthquake sequences, eastern Iran, from calibrated relocations, InSAR and high-resolution satellite imagery. 
{\it Geophysical Journal International}, {\bf 187}, 1577--1603 (2011).
}

\bibitem{ZielkeMai} Zielke O., Mai P. M. ,
Subpatch roughness in earthquake rupture investigations. {\em Geophys. Res. Lett.}  {\bf 43}, 1893--1900 (2016).

\bibitem{Zhou-Cambiotti-Sun-Sabadini} Zhou X., Cambiotti G., Sun W., Sabadini R.,
The coseismic slip distribution of a shallow subduction fault constrained by prior information: The example of 2011 Tohoku (M$_w$ 9.0) megathrust earthquake.
{\it Geophysical Journal International} {\bf 199}, 981--995 (2014).


\bibitem{Zwieten-Hanssen-Gutierrez} Van Zwieten G. J., Hanssen R. F., Guti\'errez M. A.,
Overview of a range of solution methods for elastic dislocation problems in geophysics.
{\it Journal of Geophysical Research: Solid Earth} {\bf 118}, 1721--1732 (2013). 




\end{thebibliography}
\end{document}